\documentclass[11pt]{article}

\usepackage{latexsym, amssymb, amscd, amsthm, amsxtra, amsmath,amsthm }
\usepackage{graphics, graphicx, color}
\usepackage[authoryear]{natbib}
\usepackage{ifpdf}
\usepackage[format=hang,indention=-1cm,small]{caption}
\usepackage[caption=false]{subfig}
\usepackage{multirow}
\usepackage{kotex}
\usepackage{hyperref}
\usepackage{enumitem}
\usepackage{diagbox}
\usepackage{makecell}
\usepackage{float}
\usepackage{algorithm}
\usepackage{algpseudocode}

\newtheorem{thm}{Theorem}

\newtheorem{lem}[thm]{Lemma}

\theoremstyle{definition}

\newtheorem*{defn*}{Definition}
\theoremstyle{remark}

\newtheorem{Assumption}{Assumption}

\newcommand{\norm}[1]{\left\Vert#1\right\Vert}
\newcommand{\abs}[1]{\left\vert#1\right\vert}

\newcommand{\Real}{\mathbb R}
\newcommand{\eps}{\varepsilon}

\def\argmin{\mathop{\rm argmin}}
\def\argmax{\mathop{\rm argmax}}

\def\diag{\mbox{diag}}

\def\diag{\mbox{Diag}}

\def\1v{\mathbf 1}
\def\0v{\mathbf 0}

\title{Adaptive Reference-Guided Estimation of Principal Component Subspace in High Dimensions}
\author{
  Dongsun Yoon\thanks{Department of Statistics, Seoul National University. Email: tooth0pasty@snu.ac.kr}
  \and
  Sungkyu Jung\thanks{Department of Statistics, Seoul National University. Email: sungkyu@snu.ac.kr}
}
\date{\today}
\begin{document}
\maketitle	

\begin{abstract}
We propose a novel estimator for the principal component (PC) subspace tailored to the high-dimension, low-sample size (HDLSS) context. The method, termed Adaptive Reference-Guided (ARG) estimator, is designed for data exhibiting spiked covariance structures and seeks to improve upon the conventional sample PC subspace by leveraging auxiliary information from reference vectors, presumed to carry prior knowledge about the true PC subspace. The estimator is constructed by first identifying vectors asymptotically orthogonal to the true PC subspace within a signal subspace, the subspace spanned by the leading sample PC directions and the references, and then taking the orthogonal complement. The estimator is adaptive, as it automatically selects the subspace asymptotically closest to the true PC subspace inside the signal subspace, without requiring parameter tuning. We show that when the reference vectors carry nontrivial information, the proposed estimator asymptotically reduces all principal angles between the estimated and true PC subspaces compared to the naive sample-based estimator. Interestingly, despite being derived from a completely different rationale, the ARG estimator is theoretically equivalent to an estimator based on James–Stein shrinkage. Our results thus establish a theoretical foundation that unifies these two distinct approaches.
\end{abstract}

\section{Introduction}\label{sec:intro}
High-dimension, low-sample size (HDLSS) data, introduced by \cite{Hall2005}, frequently appear in modern applications such as genetic microarrays, medical imaging, text recognition, and finance, where the number of variables $p$ far exceeds the sample size $n$. To address the challenges posed by high-dimensionality, a widely adopted strategy is to assume a spiked covariance structure \citep{Johnstone2001, Donoho2018, Ke2023}, where only a small number of eigenvalues (the ``spikes") of the true covariance matrix are significantly larger than the rest. This structure suggests that the primary signal lies in a low-dimensional true principal component (PC) subspace spanned by the corresponding eigenvectors. As a result, accurately estimating this true PC subspace is crucial, as it directly influences the performance of Principal Component Analysis (PCA) and the validity of subsequent analyses in high-dimensional settings.

Under the HDLSS asymptotic regime, where the dimension $p \to \infty$ while the sample size $n$ remains fixed, the magnitude of the spikes is naturally expressed as a function of $p$ \citep{Ahn2007, Jung2009, Jung2012, Yata2012, Shkolnik2022}. Specifically, the $i$th largest eigenvalue of the covariance matrix, denoted by $\lambda_i$, satisfies $\lambda_i = O(p^{\alpha_i})$ for $i=1,\dots,m$, where $m$ represents the number of spikes. We further assume that all spikes share the same exponent, that is, $\alpha_i = \alpha$ for $i=1,\dots,m$. The asymptotic behavior of PCA in this setting depends on the value of $\alpha$ \citep{Jung2009,Jung2012}.
We focus on the most intriguing boundary case, $\alpha = 1$, which aligns with the literature on approximate factor models \citep{Chamberlain1983, Bai2002, Fan2011}. In this setting, \cite{Jung2012} showed that the sample PC subspace is inconsistent. That is, the discrepancy between the sample and true PC subspaces, measured in terms of the principal angles, converges in probability to a random variable in $(0, \pi/2)$. This inherent inconsistency naturally leads to a fundamental question in PC subspace estimation under HDLSS settings: is it possible to construct a subspace estimator that is asymptotically closer to the true PC subspace than the conventional sample PC subspace?

In this paper, we provide a solution to the above question under the assumption that some prior information is available. This prior information takes the form of a set of unit vectors, referred to as \textit{reference directions}, which are believed to be close to the true PC subspace, although their exact alignment is unknown. A representative example is the normalized vector of ones, $(1,\dots,1)^{\top}/\sqrt{p}$, often used in financial applications to reflect the common market factor, as in the Capital Asset Pricing Model \citep{Sharpe1964, Lintner1965}.

We propose a novel estimator for the PC subspace, termed the \textit{Adaptive Reference-Guided (ARG) estimator}, which effectively incorporates the information contained in the reference directions. The term ``adaptive" reflects the estimator's ability to automatically integrate additional information from the reference directions without requiring parameter tuning, while ``reference-guided" emphasizes its use of predefined reference directions to enhance estimation accuracy. To fully leverage this additional information, we utilize \textit{negatively ridged discriminant vectors}, which naturally arise in HDLSS discrimination problems \citep{Bartlett2020,Kobak2020,Chang2021,Kim2024}. In our context, the reference directions play the role of mean difference in the discriminant vectors. As precisely defined in Section~\ref{sec:negative_ridge}, these vectors with a carefully chosen ridge parameter can be obtained purely from data and, most importantly, are asymptotically orthogonal to the true PC subspace, a perhaps counterintuitive but essential property that forms the foundation of our method.

We define the ARG subspace estimator by the orthogonal complement of the subspace spanned by the negatively ridged discriminant vectors within the \textit{signal subspace}, spanned by both the sample PC directions and the reference directions. By construction, the ARG estimator becomes the subspace asymptotically closest to the true PC subspace within the signal subspace. We show that when the reference directions carry nontrivial information about the true PC subspace, our method asymptotically reduces all principal angles between the estimated and true PC subspaces compared to the naive estimator. These theoretical findings are validated by numerical studies, demonstrating significant accuracy improvements when reference directions contain substantial information about the true PC subspace.

There is a growing body of research on utilizing prior information to improve the performance of PC subspace estimators in the HDLSS setting \citep{Goldberg2020,Shkolnik2022,Goldberg2022,Gurdogan2022,Goldberg2023,Gurdogan2024}. \cite{Shkolnik2022} examined the simplest case of a single-spike model with a single reference direction and demonstrated that shrinking (or rotating) the first sample PC direction toward the reference direction leads to a \textit{James-Stein estimator}, which achieves better asymptotic accuracy than the first sample PC direction. Expanding on this idea, \cite{Shkolnik2025} extended the James-Stein estimator to the general multi-spike model with multiple reference directions. However, their work does not provide a clear theoretical justification for why the James-Stein estimator outperforms the sample PC subspace. Additionally, it lacks an intuitive explanation of how the shrinkage framework extends to the general multi-spike model with multiple reference directions, as shrinking an $m$-dimensional sample PC subspace toward an $r$-dimensional subspace spanned by the reference directions is nontrivial.  

Although the ARG estimator is derived from a fundamentally different perspective, rooted in ridged discriminant vectors, it is, perhaps surprisingly, exactly equivalent to the James-Stein estimator. This unexpected equivalence offers deeper theoretical insight into how incorporating prior information enhances PC subspace estimation. Notably, our work fills the gaps left by \cite{Shkolnik2025}, by providing the first theoretical guarantee for the ARG estimator (or, equivalently, for the James-Stein estimator), along with an intuitive explanation that applies to the general setting of multiple spikes and multiple reference directions.

\section{Asymptotics of principal components and reference directions}

Let $X_1,\dots,X_n \in \mathbb{R}^p$ be a sample drawn from an absolutely continuous distribution in $\Real^p$ with an unknown mean vector $\mu$ and a symmetric positive-definite covariance matrix $\Sigma$. We consider the high-dimensional setting where the dimension $p$ exceeds the sample size $n$. The eigendecomposition of $\Sigma$ is given by $\Sigma = U\Lambda U^{\top} = \sum_{i=1}^p \lambda_i u_i u_i^{\top}$, where $U = [u_1,\dots,u_p]$ is the $p \times p$ matrix of eigenvectors, and $\Lambda$ is the $p \times p$ diagonal matrix of ordered eigenvalues $\lambda_1 \geq \cdots\geq \lambda_p$. We assume a spiked covariance model, where a small number $m$ of eigenvalues grow proportionally to $p$, while the remaining eigenvalues remain bounded.

\begin{Assumption} \label{assum:spiked_covariance}
    For $1 \leq m \leq n-2$, $\sigma_i^2,\tau_i^2>0$, the eigenvalues of the covariance matrix $\Sigma$ satisfy:
    \[
    \lambda_i =
    \begin{cases} 
    \sigma_i^2 p + \tau_i^2, & \quad \text{for } i = 1, \dots, m, \quad \text{(spiked eigenvalues)} \\[8pt]
    \tau_i^2, & \quad \text{for } i = m+1, \dots, p. \quad \text{(non-spiked eigenvalues)}
    \end{cases}
    \]
    Furthermore, $\max_{1\leq i \leq p}\tau_i^2$ is uniformly bounded and $\sum_{i=1}^p \tau_i^2/p \rightarrow \tau^2$ as $p \rightarrow \infty$, for some constant $\tau^2>0$.
\end{Assumption}

Denote the $p \times n$ data matrix as $X = [X_1,\dots,X_n]$ and the matrix of standardized true PC scores of $X$ as $Z$, that is,
\begin{equation*}
    Z = \Lambda^{-1/2} U^{\top} (X - \mu 1_n^{\top}) = \begin{bmatrix} z_1^{\top} \\ \vdots \\ z_p^{\top} \end{bmatrix} = \begin{bmatrix} z_{11}, \dots, z_{1n} \\ \vdots \\ z_{p1}, \dots, z_{pn} \end{bmatrix},
\end{equation*}
where $z_i$ is the $n$-dimensional vector of standardized true PC scores associated with the $i$th PC. In the HDLSS asymptotic regime, where the dimension $p$ grows while the sample size $n$ remains fixed, a key aspect of the analysis is the application of the law of large numbers across variables ($p \rightarrow \infty$), rather than across the sample ($n \rightarrow \infty$). To ensure the validity of this approach, we impose the following $\rho$-mixing condition, which regulates the dependence among the elements of a $p$-dimensional vector $(z_{1i},\dots,z_{pi})$; see \cite{Kolmogorov1960,Bradley2005}. This assumption is significantly weaker than normality or independence, yet it provides sufficient conditions for applying the law of large numbers.

\begin{Assumption} \label{assum:rho_mixing}
    The elements of the $p$-dimensional vector $(z_{1i},\dots,z_{pi})$ have uniformly bounded fourth moments. For each $p$, $(z_{1i},\dots,z_{pi})$ is a truncation of an infinite sequence $(z_{(1)},z_{(2)},\dots)_{i}$, which satisfies the $\rho$-mixing condition under some permutation.
\end{Assumption}

Analyzing high-dimensional data using PCA relies on the sample covariance matrix $S = \sum_{i=1}^n (X_i - \bar{X})(X_i - \bar{X})^{\top}/n$, where $\bar{X}$ is the sample mean. Since $S$ has rank $n-1$ almost surely, its eigendecomposition is given by $S =\hat{U}_{n-1}\hat{\Lambda}_{n-1}\hat{U}_{n-1}^{\top} = \sum_{i=1}^{n-1} \hat{\lambda}_i\hat{u}_i\hat{u}_i^{\top}$, where $\hat{U}_{n-1} = [\hat{u}_1,\dots,\hat{u}_{n-1}]$ is the $p \times (n-1)$ matrix of eigenvectors (sample PC directions), and $\hat{\Lambda}_{n-1}$ is the $(n-1) \times (n-1)$ diagonal matrix of ordered eigenvalues $\hat{\lambda}_1 \geq \cdots \geq \hat{\lambda}_{n-1}$ (sample PC variances). We choose $\hat{u}_i$ so that $\hat{u}_i^{\top} u_i \geq 0$ for $i = 1, \dots, n - 1$.

The goal of this paper is to develop an asymptotically improved estimator for the PC subspace, and it is essential to first understand the limiting behavior of the sample PC variances $\hat{\lambda}_i$ and sample PC directions $\hat{u}_i$ as $p \rightarrow \infty$. The following lemma, which is a slight modification of a result in \cite{Jung2012}, serves as a foundation for our subsequent analysis. As the proof follows directly from the argument in \cite{Jung2012}, with only minor adjustments, we omit it here.

We use the following notation. Let $\phi_i(A)$ denote the $i$th largest eigenvalue of a real symmetric matrix $A$, and let $v_i(A)$ be the corresponding eigenvector. Denote the $j$th element of $v_i(A)$ by $v_{ij}(A)$. Define the $n \times m$ matrix of the leading $m$ component scores as $W = [\sigma_1z_1,\dots,\sigma_mz_m]$, and let $\Omega = W^{\top}(I_n-J_n)W$ where $J_n = 1_n1_n^{\top}/n$. The random matrix $\Omega$ plays a central role in representing the asymptotic behavior of the sample PC variances $\hat{\lambda}_i$ and sample PC directions $\hat{u}_i$.

\begin{lem}[Theorem 2 of \cite{Jung2012}] \label{lem:asymp_properties_PC}
    Suppose Assumptions \ref{assum:spiked_covariance}-\ref{assum:rho_mixing} hold. Then, as $p \rightarrow \infty$, the sample PC variances and directions exhibit the following asymptotic behavior:

    \begin{enumerate}[label=(\roman*)]  
    
        \item Limits of sample PC variances
        \[
        \frac{n \hat{\lambda}_i}{p} \xrightarrow{P}  
        \begin{cases}  
        \phi_i(\Omega) + \tau^2, & \quad 1 \leq i \leq m, \\  
        \tau^2, & \quad m+1 \leq i \leq n-1.  
        \end{cases}  
        \]  
    
        \item Limits of sample PC directions
        \[
        \hat{u}_i^{\top} u_j \xrightarrow{P}  
        \begin{cases}  
        \sqrt{\frac{\phi_i(\Omega)}{\phi_i(\Omega) + \tau^2}} v_{ij}(\Omega), & \quad 1 \leq i \leq m, \; 1 \leq j \leq m, \\  
        0, & \quad \text{otherwise}.  
        \end{cases}  
        \]  
    
    \end{enumerate}
    
\end{lem}

    Lemma \ref{lem:asymp_properties_PC}(ii) implies that sample PC directions $\hat{u}_i$ are inconsistent, leaving room for improvement in the PC subspace estimation. To address this issue, one can incorporate prior information---external knowledge that provides guidance on the true PC subspace $\mathcal{U}_m$. In this context, prior information is represented by a set of unit vectors $\{v_1,\dots,v_r\}$, referred to as \textit{reference directions}, which are believed to be close to $\mathcal{U}_m$, though their exact alignment is unknown. Reference directions naturally arise in various applications where domain expertise suggests an approximate structure for the data. For instance, in financial applications, the normalized vector of ones, $(1,\dots,1)^{\top}/\sqrt{p}$, can be used under the assumption that all assets follow a common market factor, as suggested by various asset pricing models such as the Capital Asset Pricing Model \citep{Sharpe1964,Lintner1965} and the Fama-French three- and five-factor models \citep{Fama1993,Fama2015}. Similarly, in genetics, reference directions may be derived from known biological pathways \citep{Kanehisa2000}, where groups of functionally related genes exhibit coordinated expression patterns, providing prior information about gene expression data. To formally incorporate reference directions into the PC subspace estimation framework, we impose the following conditions on their structure and alignment with $\mathcal{U}_m$. Denote the $p \times r$ matrix of reference directions as $V_r = [v_1,\dots,v_r]$ and the subspace spanned by reference directions as $\mathcal{V}_r = \operatorname{span}(v_1,\dots,v_r)$.

\begin{Assumption} \label{assum:reference_directions_non_singular}
    For $1 \leq r \leq p-m$, the reference directions $v_1,\dots,v_r$ are linearly independent $p$-dimensional unit vectors. The $r \times r$ Gram matrix $V_r^{\top}V_r$ converges entrywise to a non-singular matrix $V^{G}$ as $p \rightarrow \infty$.
\end{Assumption}
Assumption \ref{assum:reference_directions_non_singular} ensures that the reference directions $v_1,\dots,v_r$ span a well-conditioned $r$-dimensional subspace $\mathcal{V}_r$, preventing numerical instabilities in subsequent analysis. Additionally, we quantify the relationship between the reference directions and the true PC subspace $\mathcal{U}_m$ through their projections and inner products. Let $P_{\mathcal{T}}v$ denote the orthogonal projection of a vector $v$ onto a subspace $\mathcal{T}$.

\begin{Assumption} \label{assum:reference_directions_ratios}
    There exist constants $a_i \in [0,1]$ such that $\norm{P_{\mathcal{U}_m}v_i} \rightarrow a_i$ as $p \rightarrow \infty$ for $i=1,\dots,r$. Furthermore, there exist constants $a_{ij} \in [-1,1]$ such that $v_i^{\top}u_j \rightarrow a_{ij}$ as $p \rightarrow \infty$, for $i=1,\dots,r$ and $j=1,\dots,m$.
\end{Assumption}
Assumption \ref{assum:reference_directions_ratios} quantifies the informativeness of reference directions by measuring their alignment with the true PC subspace $\mathcal{U}_m$. The quantity $a_i$ represents the asymptotic norm of the projection of $v_i$ onto $\mathcal{U}_m$, indicating how much information $v_i$ retains about $\mathcal{U}_m$. If $a_i$ is close to 1, then $v_i$ is highly informative, whereas if $a_i$ is close to 0, it provides little information. Similarly, the constants $a_{ij}$ capture the limiting inner products between the reference directions and the true PC directions, which provides a finer characterization of their directional alignment.

If the reference directions $v_1,\dots,v_r$ are believed to contain prior information about the true PC subspace $\mathcal{U}_m$, they can be used to build an improved estimator of $\mathcal{U}_m$. To effectively incorporate this information, it is essential to understand how the reference directions interact with the sample PC directions asymptotically. The following lemma, which is a slight modification of Lemma C.3 in \cite{Chang2021}, characterizes the limiting behavior of the inner products between the sample PC directions and the reference directions. This result provides a key asymptotic relationship for developing a refined estimator of $\mathcal{U}_m$. As the proof follows directly from the argument in \cite{Chang2021}, with only minor adjustments, we omit it here.

\begin{lem} [Theorem C.3 of \cite{Chang2021}] \label{lem:asymp_properties_reference_directions}
    Suppose Assumptions \ref{assum:spiked_covariance}, \ref{assum:rho_mixing}, and \ref{assum:reference_directions_ratios} hold. Then, as $p \rightarrow \infty$, the reference directions exhibit the following asymptotic behavior:
    \[
    \hat{u}_i^{\top} v_j \xrightarrow{P}  
    \begin{cases}  
        \sqrt{\frac{\phi_i(\Omega)}{\phi_i(\Omega) + \tau^2}} \sum_{k=1}^m a_{jk} v_{ik}(\Omega), & \quad 1 \leq i \leq m, \; 1 \leq j \leq r, \\  
        0, & \quad m+1 \leq i \leq n-1, \; 1 \leq j \leq r.  
    \end{cases}  
    \]

\end{lem}

\section{The Adaptive Reference-Guided estimator}

\subsection{Negatively ridged discriminant vectors} \label{sec:negative_ridge}

To improve upon the naive sample PC subspace $\hat{\mathcal{U}}_m = \operatorname{span}(\hat{u}_1,\dots,\hat{u}_m)$ in estimating the true PC subspace $\mathcal{U}_m$ using the prior information from the reference directions, it is natural to consider the following candidate subspace:
\begin{equation*}
    \mathcal{S}:=\operatorname{span}(\hat{u}_1,\dots,\hat{u}_m,v_1,\dots,v_r),
\end{equation*}
with the goal of identifying a subspace within $\mathcal{S}$ that is asymptotically closest to $\mathcal{U}_m$. We refer to $\mathcal{S}$ as the \textit{signal subspace}, since it consists of both the signal part of the sample PC directions and the reference directions. The remaining directions, $\hat{u}_{m+1},\dots,\hat{u}_{n-1}$, represent noise components that do not contain useful information about $\mathcal{U}_m$, as shown in Lemma \ref{lem:asymp_properties_PC}(ii). Thus, we discard them. The signal subspace $\mathcal{S}$ is $(m+r)$-dimensional almost surely, since $\hat{\mathcal{U}}_m \cap \mathcal{V}_r=\{0_p\}$ holds almost surely for data sampled from absolutely continuous distribution with $m+r \leq p$.

The subspace closest to $\mathcal{U}_m$ inside $\mathcal{S}$ is the orthogonal projection $P_{\mathcal{S}}\mathcal{U}_m = \operatorname{span}(P_{\mathcal{S}}u_1,\dots,P_{\mathcal{S}}u_m)$. However, since $P_{\mathcal{S}}\mathcal{U}_m$ depends on the unknown true PC directions $u_1,\dots,u_m$, it cannot be directly used. 
Instead, we aim to construct an estimator that consistently approximates $P_{\mathcal{S}}\mathcal{U}_m$. As a first step, we define the \textit{negatively ridged discriminant vectors} $d_1,\dots,d_r$ as follows:
\begin{equation}
    d_i := -\tilde{\lambda}(S_m-\tilde{\lambda}I_p)^{-1}v_i,
    \label{eq:negative_ridge}
\end{equation}
where $\tilde{\lambda} = \sum_{i=m+1}^{n-1}\hat{\lambda}_i/(n-m-1)$ is the average of non-spiked sample eigenvalues and $S_m = \sum_{i=1}^m \hat{\lambda}_i\hat{u}_i\hat{u}_i^{\top}$ is the spiked component of the sample covariance matrix $S$. We refer to these vectors as the negatively ridged discriminant vectors, since they take the form of ridged discriminant vectors with covariance $S_m$, mean differences $v_i$, and ridge parameter $-\tilde{\lambda}$. The inverse $(S_m-\tilde{\lambda}I_p)^{-1}$ is well defined almost surely, since $S_m-\tilde{\lambda}I_p$ is non-singular if and only if $\hat{\Lambda}_m^{-1} - I_m/\tilde{\lambda} = \diag(\hat{\lambda}_i^{-1} - \tilde{\lambda}^{-1})$ is non-singular. The only case where $\diag(\hat{\lambda}_i^{-1} - \tilde{\lambda}^{-1})$ becomes singular is when $\hat{\lambda}_m = \tilde{\lambda}$, which is equivalent to $\hat{\lambda}_m = \cdots = \hat{\lambda}_{n-1}$, an event that occurs with probability zero under absolutely continuous distributions. Applying the Sherman-Morrison-Woodbury formula, we obtain an alternative expression for $d_i$:
\begin{equation}\label{eq:negative_ridge_equivalence}
    d_i = -\sum_{j=1}^m \frac{\tilde{\lambda}\hat{u}_j^{\top}v_i}{\hat{\lambda}_j-\tilde{\lambda}}\hat{u}_j + \tilde{v}_i,
\end{equation}
where $\tilde{v}_i = (I_p - \hat{U}_m\hat{U}_m^{\top})v_i$ denotes the component of $v_i$ orthogonal to the sample PC subspace $\hat{\mathcal{U}}_m$.

Collecting the vectors introduced in (\ref{eq:negative_ridge_equivalence}), write $D_r = [d_1,\dots,d_r]$ and $\tilde{V}_r = [\tilde{v}_1,\dots,\tilde{v}_r]$. Denote the subspaces spanned by these vectors as $\mathcal{D}_r = \operatorname{span}(d_1,\dots,d_r)$ and $\tilde{\mathcal{V}}_r = \operatorname{span}(\tilde{v}_1,\dots,\tilde{v}_r)$, respectively. 
It follows that both $\mathcal{D}_r$ and $\tilde{\mathcal{V}}_r$ are $r$-dimensional subspaces contained within the signal subspace $\mathcal{S}$. For $\mathcal{D}_r$, we have $\dim(\mathcal{D}_r)=r$ due to the non-singularity of $S_m - \tilde{\lambda}I_p$ and the linear independence of reference directions, while the inclusion $\mathcal{D}_r \subset \mathcal{S}$ follows from \eqref{eq:negative_ridge_equivalence}. Similarly, for $\tilde{\mathcal{V}}_r$, we have $\dim(\tilde{\mathcal{V}}_r)=r$ since $\dim(\mathcal{S})=m+r$ and the inclusion $\tilde{\mathcal{V}}_r \subset \mathcal{S}$ is immediate by the definition of $\tilde{v}_i$. The following Lemma further asserts that both $\mathcal{D}_r$ and $\tilde{\mathcal{V}}_r$ are well-conditioned $r$-dimensional subspaces, a fact that will be used in the proof of our main theorem.

\begin{lem} \label{lem:well_conditioned_subspaces}
    Under Assumptions \ref{assum:spiked_covariance}--\ref{assum:reference_directions_ratios}, as $p \rightarrow \infty$, the Gram matrices $D_r^{\top}D_r$ and $\tilde{V}_r^{\top}\tilde{V}_r$ converge entrywise in probability to some random matrices $D^G$ and $\tilde{V}^G$, respectively. Moreover, $D^G$ and $\tilde{V}^G$ are almost surely non-singular.
\end{lem}

\begin{proof}[Proof of Lemma \ref{lem:well_conditioned_subspaces}]
    We first show that $D_r^{\top}D_r$ and $\tilde{V}_r^{\top}\tilde{V}_r$ converge entrywise in probability to $r \times r$ random matrices, which we denote by $D^G$ and $\tilde{V}^G$, respectively. Using the equivalent definition of $d_i$ in \eqref{eq:negative_ridge_equivalence}, the inner product between two negatively ridged discriminant vectors is given by:

    \begin{equation} \label{eq:negative_ridge_inner_product}
        d_i^{\top}d_j = \sum_{k=1}^m \frac{\tilde{\lambda}^2(\hat{u}_k^{\top}v_j)(\hat{u}_k^{\top}v_i)}{(\hat{\lambda}_i-\tilde{\lambda})(\hat{\lambda}_j-\tilde{\lambda})} + v_i^{\top}(I_p - \hat{U}_m\hat{U}_m^{\top})v_j.
    \end{equation}
    Since $v_i^{\top}v_j$ converges to $V_{ij}^G$ due to Assumption \ref{assum:reference_directions_non_singular} and the remaining parts of \eqref{eq:negative_ridge_inner_product} converge in probability due to Lemmas \ref{lem:asymp_properties_PC} and \ref{lem:asymp_properties_reference_directions}, we conclude that $d_i^{\top}d_j$ converges in probability. Since $\tilde{v}_i^{\top}\tilde{v}_j$ is merely the second term of (\ref{eq:negative_ridge_inner_product}), $\tilde{v}_i^{\top}\tilde{v}_j$ also converges in probability.
    
    Denote the smallest singular value of the matrix $M$ as $\sigma_{min}(M)$. We will show that $D^G$ and $\tilde{V}^G$ are almost surely non-singular by verifying that $\sigma_{min}(D^G)>0$ and $\sigma_{min}(\tilde{V}^G)>0$ almost surely. Since the smallest singular value of the product of two matrices is at least the product of their smallest singular values, we obtain the following inequality:

    \begin{equation*}
        \sigma_{min}(D_r) \geq \sigma_{min}(-\tilde{\lambda}(S_m - \tilde{\lambda}I_p)^{-1}) \cdot \sigma_{min}(V_r).
    \end{equation*}
    By Assumption \ref{assum:reference_directions_non_singular}, $\sigma_{min}(V_r)$ converges to a positive value $\sqrt{\sigma_{min}(V^G)}$. Moreover, $\sigma_{min}(-\tilde{\lambda}(S_m - \tilde{\lambda}I_p)^{-1}) = \tilde{\lambda}/(\hat{\lambda}_1 - \tilde{\lambda}) \wedge 1$, which converges in probability to a positive random variable $\tau^2/\phi_1(\Omega) \wedge 1$ due to Lemma \ref{lem:asymp_properties_PC}(i). Therefore, the probability limit of $\sigma_{min}(D_r)$, $\sqrt{\sigma_{min}(D^G)}$, is greater than zero almost surely.

    Now consider the smallest singular value of $\tilde{V}_r$. Since $\tilde{V}_r = \hat{U}_{-m}\hat{U}_{-m}^{\top}V_r$, where $\hat{U}_{-m} = [\hat{u}_{m+1},\dots,\hat{u}_p]$ and $\hat{u}_{n},\dots,\hat{u}_p$ are arbitrary orthonormal basis of $\Real^p \setminus \hat{U}_{n-1}$, we obtain $\sigma_{min}(\tilde{V}_r) = \sigma_{min}(\hat{U}_{-m}^{\top}V_r)$. Since $\sigma_{min}(\hat{U}_{-m})=1$, the inequality of the smallest singular values implies $\sigma_{min}(\tilde{V}_r) \geq \sigma_{min}(V_r)$. Therefore, the probability limit of $\sigma_{min}(\tilde{V}_r)$, $\sqrt{\sigma_{min}(\tilde{V}^G)}$, is greater than zero almost surely due to Assumption \ref{assum:reference_directions_non_singular}.
\end{proof}

A key idea in defining our estimator is that the negatively ridged discriminant vectors $d_1,\dots,d_r$ are asymptotically orthogonal to the true PC subspace $\mathcal{U}_m$. At first glance, this result may seem counterintuitive: while $\mathcal{U}_m$ is an unknown parameter,  $d_1,\dots,d_r$ are obtained entirely from data. 
The intuition behind this phenomenon lies in the structure of the signal subspace $\mathcal{S}$. Within $\mathcal{S}$, it is possible to construct a vector that preserves the inner product between the true PC direction $u_i$ and the reference direction $v_j$, by compensating for the loss induced by the discrepancy between $\mathcal{U}_m$ and the sample PC subspace $\hat{\mathcal{U}}_m$. These insights are central to the proof of the next result.


\begin{thm} \label{thm:orthogonality}
    Under Assumptions \ref{assum:spiked_covariance}, \ref{assum:rho_mixing}, and \ref{assum:reference_directions_ratios}, as $p \rightarrow \infty$, $\operatorname{Angle}(u_i,d_j) \xrightarrow{P} \pi/2$ for $i=1,\dots,m$ and $j = 1,\dots,r$.
\end{thm}

\begin{proof}[Proof of Theorem \ref{thm:orthogonality}]

    We first prove that, for $i=1,\dots,m$ and $j=1,\dots,r$,
    
    \begin{equation} \label{eq:inner_product}
        u_i^{\top}d_j = u_i^{\top}v_j - u_i^{\top}\sum_{k=1}^m \frac{\hat{\lambda}_k\hat{u}_k^{\top}v_j}{\hat{\lambda}_k - \tilde{\lambda}}\hat{u}_k
    \end{equation}
    converges in probability to 0 as $p \rightarrow \infty$. The first term $u_i^{\top}v_j$ of \eqref{eq:inner_product} converges to $a_{ji}$ by Assumption \ref{assum:reference_directions_ratios}. We will show that the second term, $u_i^{\top}\sum_{k=1}^m \frac{\hat{\lambda}_k\hat{u}_k^{\top}v_j}{\hat{\lambda}_k - \tilde{\lambda}}\hat{u}_k$, also converges in probability to $a_{ji}$.

    Consider the orthogonal projection of $v_j$ onto $\hat{\mathcal{U}}_m$, $P_{\hat{\mathcal{U}}_m}v_j = \sum_{k=1}^m (\hat{u}_k^{\top}v_j)\hat{u}_k$. Although $u_i^{\top}v_j$ converges to $a_{ji}$, due to the discrepancy between the true and sample PC subspaces, the probability limit of $u_i^{\top}P_{\hat{\mathcal{U}}_m}v_j = u_i^{\top}\sum_{k=1}^m(\hat{u}_k^{\top}v_j)\hat{u}_k$ should be smaller than the desired value $a_{ji}$. Applying Lemmas \ref{lem:asymp_properties_PC} and \ref{lem:asymp_properties_reference_directions}, we obtain the following asymptotic property:

    \begin{equation} \label{eq:without_scaling}
        u_i^{\top}\sum_{k=1}^m (\hat{u}_k^{\top}v_j)\hat{u}_k\xrightarrow{P} \sum_{k=1}^m \sum_{l=1}^m  \frac{\phi_k(\Omega)}{\phi_k(\Omega) + \tau^2}  a_{jl} v_{kl}(\Omega) v_{ki}(\Omega).
    \end{equation}
    
    The value $\frac{\phi_k(\Omega)}{\phi_k(\Omega) + \tau^2}$ can be interpreted as the proportion of the inner product preserved by the projection onto $\hat{u}_k$. To see this, note that removing the factor $\frac{\phi_k(\Omega)}{\phi_k(\Omega) + \tau^2}$ from the right-hand side of \eqref{eq:without_scaling} yields

    \begin{equation*}
        \sum_{k=1}^m \sum_{l=1}^m a_{jl} v_{kl}(\Omega) v_{ki}(\Omega) = a_{ji},
    \end{equation*}
    which is exactly the desired inner product value. The equality follows from the fact that $V(\Omega) = [v_1(\Omega),\dots,v_m(\Omega)]$ is an orthonormal matrix.
    
    Then, by scaling each component $(\hat{u}_k^{\top}v_j)\hat{u}_k$ by $\frac{\phi_k(\Omega) + \tau^2}{\phi_k(\Omega)}$, the inverse of the proportion of the inner product preserved, we obtain a vector that satisfies the desired asymptotic property. Although $\frac{\phi_k(\Omega) + \tau^2}{\phi_k(\Omega)}$ cannot be computed from data, Lemma \ref{lem:asymp_properties_PC}(i) implies that it can be consistently estimated by $\frac{\hat{\lambda}_k}{\hat{\lambda}_k - \tilde{\lambda}}$. Therefore, the scaled version of the second term in \eqref{eq:inner_product}, given by $u_i^{\top}\sum_{k=1}^m \frac{\hat{\lambda}_k\hat{u}_k^{\top}v_j}{\hat{\lambda}_k - \tilde{\lambda}}\hat{u}_k$, converges in probability to $a_{ji}$ as $p \rightarrow \infty$.

    To complete the proof, we show that $\norm{d_j}^2$ does not degenerate asymptotically. Due to the equivalent definition of $d_j$ in \eqref{eq:negative_ridge_equivalence}, it is enough to show that $\norm{\tilde{v}_j}^2$ converges in probability to a strictly positive random variable. Applying Lemmas \ref{lem:asymp_properties_PC} and \ref{lem:asymp_properties_reference_directions}, we obtain the following asymptotic property:
    \begin{equation} \label{eq:v_tilde_norm_asymptotics}
        \norm{\tilde{v}_j}^2 = 1-v_j^{\top}\hat{U}_m\hat{U}_m^{\top}v_j \xrightarrow{P} 1 - \sum_{i=1}^m \frac{\phi_i(\Omega)}{\phi_i(\Omega) + \tau^2} \left(\sum_{k=1}^ma_{jk} v_{ik}(\Omega)\right)^2
    \end{equation}
    Let $a_j^{vec} = (a_{j1},\dots,a_{jm})^{\top}$. Since $V(\Omega)$ is an orthogonal matrix and $\norm{a_j^{vec}}^2 = a_j^2 \leq 1$,
    \begin{equation*}
        \sum_{i=1}^m \left(\sum_{k=1}^ma_{jk} v_{ik}(\Omega)\right)^2 = \norm{V(\Omega)^{\top}a_j^{vec}}^2 = \norm{a_j^{vec}}^2 \leq 1,
    \end{equation*}
    which implies that the second term of probability limit in \eqref{eq:v_tilde_norm_asymptotics} is strictly less than one since $\tau^2>0$.
\end{proof}

\subsection{The single-spike, single-reference case}

Before turning to the general setting involving multiple spikes and multiple reference directions, we begin by examining the simplest scenario of a single-spike and a single-reference direction. A detailed investigation of this basic case provides deeper insight into the intuition and properties of the propposed estimator. 

We define the \textit{Adaptive Reference-Guided (ARG) estimator} specifically for the single-spike, single-reference case as the orthogonal complement of the negatively ridged discriminant vector $d_1$ within the signal subspace $\mathcal{S} = \operatorname{span}(\hat{u}_1,v_1)$:

\begin{equation*}
    \hat{u}_1^{ARG} := \mathcal{S} \setminus d_1.
\end{equation*}

Figure \ref{fig:idea} illustrates the core idea behind the ARG estimator. By Theorem \ref{thm:orthogonality}, the negatively ridged discriminant vector $d_1$ is asymptotically orthogonal to the true PC direction $u_1$. Therefore, we define $\hat{u}_1^{ARG}$ as the orthogonal complement of $d_1$ within the signal subspace $\mathcal{S}$, ensuring that $\hat{u}_1^{ARG}$ asymptotically aligns with the orthogonal projection $P_{\mathcal{S}}u_1$.

\begin{figure}
    \centering
    \includegraphics[width=0.63\linewidth]{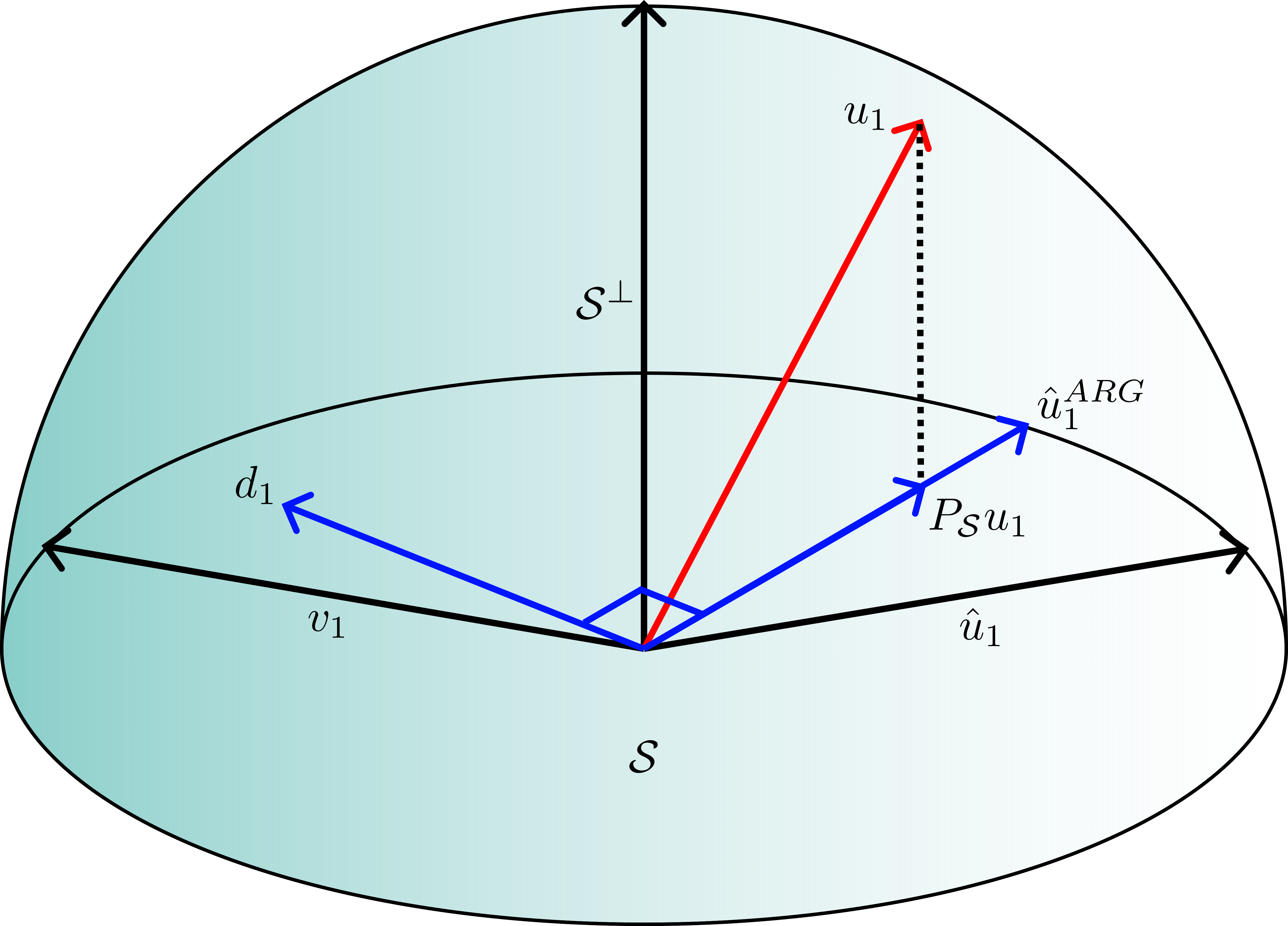}
    \caption{Geometric representation of the ARG estimator in the single-spike, single-reference setting. The signal subspace $\mathcal{S} = \operatorname{span}(\hat{u}_1,v_1)$ is depicted as the $x-y$ plane, while the $z$-axis represents the component orthogonal to $\mathcal{S}$. The negatively ridged discriminant vector $d_1$ (blue) is asymptotically orthogonal to the true PC direction $u_1$ (red). The ARG estimator $\hat{u}_1^{ARG}$ (blue) is then defined as the orthogonal complement of $d_1$ within $\mathcal{S}$, aligning asymptotically with $P_{\mathcal{S}}u_1$ (blue). Notably, $\operatorname{Angle}(\hat{u}_1^{ARG},u_1)$ is asymptotically smaller than $\operatorname{Angle}(\hat{u}_1,u_1)$, demonstrating the improved asymptotic efficiency of the ARG estimator.}
    \label{fig:idea}
\end{figure}

Using the equivalent definition of negatively ridged discriminant vector $d_1$ in \eqref{eq:negative_ridge_equivalence}, we obtain a closed-form expression for $\hat{u}_1^{ARG}$:
\begin{equation} \label{eq:u1_ARG}
    \hat{u}_1^{ARG} \propto  \left(\frac{\hat{\lambda}_1}{\tilde{\lambda}}(1-(\hat{u}_1^{\top}v_1)^2)-1\right)\hat{u}_1 + (\hat{u}_1^{\top}v_1)v_1.
\end{equation}
This expression exactly coincides with the James-Stein estimator that shrinks the first eigenvector $\hat{u}_1$ towards the reference direction $v_1$,  
proposed in 
\cite{Shkolnik2022}. An intuitive explanation for the equivalence between the ARG and James-Stein estimators will be provided in Section \ref{sec:general_case}, in the context of the general multi-spike, multi-reference scenario.

We next examine how the amount of information $a_1$ influences the ARG estimator $\hat{u}_1^{ARG}$. Recall that $a_1 = \lim_{p\to\infty} |u_1^\top v_1|$ (Assumption~\ref{assum:reference_directions_ratios}). If $a_1=0$, indicating that the reference direction $v_1$ carries no information about the true PC direction $u_1$, then Lemma \ref{lem:asymp_properties_reference_directions} implies that the coefficient of $v_1$ in \eqref{eq:u1_ARG}, $\hat{u}_1^{\top}v_1$, converges in probability to zero. Consequently, $\hat{u}_1^{ARG} \approx \hat{u}_1$ holds asymptotically, indicating that the ARG estimator yields no improvement when no information is available. On the other hand, if $a_1=1$, meaning that $v_1$ perfectly aligns with $u_1$ in the limit, applying Lemmas \ref{lem:asymp_properties_PC} and \ref{lem:asymp_properties_reference_directions} shows that the coefficient of $\hat{u}_1$ in \eqref{eq:u1_ARG}, $\frac{\hat{\lambda}_1}{\tilde{\lambda}}(1-(\hat{u}_1^{\top}v_1)^2)-1$, converges in probability to zero. This implies that $\hat{u}_1^{ARG} \approx v_1 \approx u_1$ asymptotically, demonstrating that full information leads to full recovery of the true PC direction. These observations highlight the crucial role of the information quantity $a_1$ in the performance of the ARG estimator. This relationship is formally stated in the following theorem, which characterizes the limiting behavior of the inner product between $\hat{u}_1^{ARG}$ and $u_1$.

\begin{thm} \label{thm:u1_asymptotic}
    Suppose Assumptions \ref{assum:spiked_covariance}, \ref{assum:rho_mixing} and \ref{assum:reference_directions_ratios} hold with $m=r=1$. Then, as $p \rightarrow \infty$, the ARG estimator $\hat{u}_1^{ARG}$ exhibits the following asymptotic behavior:
    \[
        (\hat{u}_1^{ARG})^{\top} u_1 \xrightarrow{P} \sqrt{\left(\frac{\Omega}{\Omega + \tau^2}\right) \cdot \left(1+\frac{\tau^4 a_1^2}{\Omega^2(1-a_1^2) + \Omega\tau^2}\right)}.
    \]
\end{thm}

\begin{proof} [Proof of Theorem \ref{thm:u1_asymptotic}]
    The result follows directly from Lemmas \ref{lem:asymp_properties_PC} and \ref{lem:asymp_properties_reference_directions}. Noting that $\Omega$ is scalar when $m = 1$, we simplify $\phi_1(\Omega) = \Omega$ and $v_{11}(\Omega) = 1$, which completes the proof.
\end{proof}

Combining Theorem \ref{thm:u1_asymptotic} and Lemma \ref{lem:asymp_properties_PC}(ii), we obtain the following asymptotic ratio between $(\hat{u}_1^{ARG})^{\top}u_1$ and $\hat{u}_1^{\top}u_1$:

\begin{equation*}
    \frac{(\hat{u}_1^{ARG})^{\top}u_1}{\hat{u}_1^{\top}u_1} \xrightarrow{P} \sqrt{1+\frac{\tau^4 a_1^2}{\Omega^2(1-a_1^2) + \Omega\tau^2}}.
\end{equation*}
The term inside the square root is at least 1. Hence, $\hat{u}_1^{ARG}$ is asymptotically closer to the true PC direction $u_1$ than the sample direction $\hat{u}_1$. Moreover, this quantity is an increasing function of both $\tau^2$ and $a_1$, but a decreasing function of $\Omega = \sigma_1^2z_1^{\top}(I_n - J_n)z_1$, indicating that the improvement is more significant when:

\begin{enumerate}[label=(\roman*)]
    \item The reference direction $v_1$ contains more information about the true PC direction $u_1$ (i.e., larger $a_1$).
    \item The scaled signal-to-noise ratio $\sigma_1^2/\tau^2$ is smaller.
\end{enumerate}

While the preceding analysis focused on the single-spike, single-reference setting, similar intuition extends to the general multi-spike, multi-reference situation. In this broader setting, we rigorously establish that the ARG estimator outperforms the naive sample PC subspace. Although characterizing the precise conditions under which it achieves greater accuracy remains intractable, it is natural to expect that the performance of the ARG estimator improves as the reference directions contain more prior information about the true PC subspace. We present a more detailed theoretical analysis of the general setting in the next section.

\subsection{The multi-spike, multi-reference case} \label{sec:general_case}

In this section, we define the Adaptive Reference-Guided (ARG) estimator for the general multi-spike, multi-reference case and demonstrate that it achieves improved asymptotic efficiency compared to the naive sample PC subspace $\hat{\mathcal{U}}_m$, provided that the reference directions $v_1,\dots,v_r$ carry nontrivial information about the true PC subspace $\mathcal{U}_m$. As in the single-spike, single-reference direction case, the ARG estimator is defined as the orthogonal complement of the negatively ridged discriminant subspace $\mathcal{D}_r$ within the signal subspace $\mathcal{S}$:
\begin{equation}
    \hat{\mathcal{U}}_m^{ARG} := \mathcal{S} \setminus \mathcal{D}_r.
\label{eq:ARGgen}
\end{equation}
The key intuition behind this estimator stems from the fact that the negatively ridged discriminant vectors $d_1,\dots,d_r$ are asymptotically orthogonal to $\mathcal{U}_m$, as established in Theorem \ref{thm:orthogonality}. This crucial property ensures that $\hat{\mathcal{U}}_m^{ARG}$ consistently approximates $P_{\mathcal{S}}\mathcal{U}_m$, making it the most asymptotically accurate subspace estimate of $\mathcal{U}_m$ within $\mathcal{S}$.

Next, we present an equivalent formulation of the ARG estimator $\hat{\mathcal{U}}_m^{ARG}$, expressed in terms of a non-orthonormal basis. In contrast to the original definition in (\ref{eq:ARGgen}), which is based on taking the orthogonal complement, this alternative formulation enables direct computation by explicitly providing a basis matrix for the estimated subspace. This representation arises naturally from the definition of the negatively ridged discriminant vectors $d_1,\dots,d_r$ (see (\ref{eq:negative_ridge})), yielding the following identity:
\begin{equation*}
    D_r^{\top}(S_m - \tilde{\lambda}I_p)(I_p-P_{\mathcal{V}_r})\hat{U}_m=0_{r \times m}.
\end{equation*}
This implies that the columns of the $p \times m$ matrix $(S_m - \tilde{\lambda}I_p)(I_p-P_{\mathcal{V}_r})\hat{U}_m$ are orthogonal to the negatively ridged discriminant subspace $\mathcal{D}_r$. The matrix $(S_m - \tilde{\lambda}I_p)(I_p - P_{\mathcal{V}_r})\hat{U}_m$ spans an $m$-dimensional subspace contained in the signal subspace $\mathcal{S}$, owing to the non-singularity of $S_m - \tilde{\lambda}I_p$ and the assumption that the data are sampled from an absolutely continuous distribution. Hence, we obtain the following equivalent definition of the ARG estimator:
\begin{equation*}
    \hat{\mathcal{U}}_m^{ARG} = \operatorname{span}((S_m - \tilde{\lambda}I_p)(I_p-P_{\mathcal{V}_r})\hat{U}_m).
\end{equation*}

To our surprise, the ARG estimator $\hat{\mathcal{U}}_m^{ARG}$ turns out to be exactly equivalent to the James-Stein estimator for the general multi-spike, multi-reference case, proposed in \cite{Shkolnik2025}. This equivalence arises naturally from the fact that both estimators aim to find the subspace within the signal subspace $\mathcal{S}$ that is asymptotically closest to the true PC subspace $\mathcal{U}_m$. The ARG estimator achieves this by taking the orthogonal complement of the negatively ridged discriminant subspace $\mathcal{D}_r$, utilizing the asymptotic orthogonality of Theorem \ref{thm:orthogonality}. On the other hand, the James-Stein estimator applies a shrinkage transformation to the sample PC subspace $\hat{\mathcal{U}}_m$, pulling it toward the subspace spanned by the reference directions $v_1,\dots,v_r$. Since shrinkage is a linear transformation within $\mathcal{S}$, it must yield the same subspace as the ARG estimator, which is the closest possible approximation of $\mathcal{U}_m$ inside $\mathcal{S}$ asymptotically.

To analyze the asymptotic efficiency of the ARG estimator $\hat{\mathcal{U}}_m^{ARG}$, we begin by developing the technical tools necessary to describe the asymptotic relationship between the singular values of two uniformly bounded sequences of random matrices. Let $\sigma_k(M)$ denote the $k$th largest singular value of a matrix $M$, and let $S^{m-1}$ denote the unit sphere in $\mathbb{R}^m$.

\begin{lem} \label{lem:singular_value_asymptotics}
    Let $\{A_p\}_{p=1}^\infty$ and $\{B_p\}_{p=1}^\infty$ be sequences of $m \times m$ random matrices satisfying $\norm{A_p}_2 \leq 1$ and $\norm{B_p}_2 \leq 1$ for all $p$. Then, as $p \rightarrow \infty$, the following asymptotic properties hold:
    
    \begin{enumerate}[label=(\roman*)]
        \item If $\norm{A_pc} - \norm{B_pc} \xrightarrow{P} X_c$ as $p \rightarrow \infty$ holds for all $c \in S^{m-1}$, where $\inf_{c \in S^{m-1}} X_c > 0$ almost surely, then $\mathbb{P}(\sigma_k(A_p) > \sigma_k(B_p)) \rightarrow 1$ for $k=1,\dots,m$.
        \item If $\norm{A_pc} - \norm{B_pc} \xrightarrow{P} X_c$ as $p \rightarrow \infty$ holds for all $c \in S^{m-1}$, where $X_c>0$ for all $c \in S^{m-1}$ and $\inf_{c \in S^{m-1}} X_c = 0$ almost surely, then $\mathbb{P}(\sigma_k(A_p) > \sigma_k(B_p) - \eps) \rightarrow 1$ for any $\eps>0$ and $k=1,\dots,m$.
        \item If $\norm{A_pc} - \norm{B_pc} \xrightarrow{P}0$ as $p \rightarrow \infty$ holds for all $c \in S^{m-1}$, then $\sigma_k(A_p) - \sigma_k(B_p) \xrightarrow{P}0$ for $k=1,\dots,m$.
    \end{enumerate}

\end{lem}

A proof of Lemma \ref{lem:singular_value_asymptotics} is provided in the \nameref{sec:supporting_information}. We now state our main result, which establishes the asymptotic efficiency of the ARG estimator $\hat{\mathcal{U}}_m^{ARG}$. Specifically, we show that when the reference directions $v_i$ carry nontrivial information about the true PC subspace $\mathcal{U}_m$, $\hat{\mathcal{U}}_m^{ARG}$ asymptotically achieves a more accurate estimate of $\mathcal{U}_m$ compared to the naive sample PC subspace $\hat{\mathcal{U}}_m$. In contrast, when the reference directions contain no information, $\hat{\mathcal{U}}_m^{ARG}$ offers no additional improvement over $\hat{\mathcal{U}}_m$. To formalize this result, we compare the principal angles between $\hat{\mathcal{U}}_m^{ARG}$ and $\mathcal{U}_m$ with those between $\hat{\mathcal{U}}_m$ and $\mathcal{U}_m$ in the asymptotic regime. Let $\theta_k(\mathcal{T}_1,\mathcal{T}_2)$ denote the $k$th principal angle between subspaces $\mathcal{T}_1$ and $\mathcal{T}_2$, and let $A = [a_1^{vec},\dots,a_r^{vec}]$ ($a_j^{vec} = (a_{j1},\dots,a_{jm})^{\top}$) be the $m \times r$ matrix, whose columns indicate the alignment between the reference directions $v_j$ and the true PC directions $u_i$.

\begin{thm} \label{thm:efficiency}
    Suppose Assumptions \ref{assum:spiked_covariance}-\ref{assum:reference_directions_ratios} hold. Then, as $p \rightarrow \infty$, the asymptotic relationship between $\theta_k(\hat{\mathcal{U}}_m^{ARG},\mathcal{U}_m)$ and $\theta_k(\hat{\mathcal{U}}_m,\mathcal{U}_m)$ depends on the properties of the matrix $A$.
    \begin{enumerate}[label=(\roman*)]
        \item If $A$ is a full-rank matrix, then $\mathbb{P}\big(\theta_{k}(\hat{\mathcal{U}}_m^{ARG}, \mathcal{U}_m) < \theta_{k}(\hat{\mathcal{U}}_m, \mathcal{U}_m)\big) \rightarrow 1$ for $k = 1,\dots,m$.
        \item If $A$ is neither a full-rank matrix nor the zero matrix, then $\mathbb{P}\big(\theta_{k}(\hat{\mathcal{U}}_m^{ARG}, \mathcal{U}_m) < \theta_{k}(\hat{\mathcal{U}}_m, \mathcal{U}_m) + \eps\big) \rightarrow 1$ for any $\eps>0$ and $k = 1,\dots,m$.
        \item If $A$ is the zero matrix, that is, the reference directions contain no information about $\mathcal{U}_m$, then $\theta_{k}(\hat{\mathcal{U}}_m^{ARG}, \mathcal{U}_m) - \theta_{k}(\hat{\mathcal{U}}_m, \mathcal{U}_m) \xrightarrow{P} 0$ for $k = 1,\dots,m$.
    \end{enumerate}
\end{thm}

\begin{proof} [Proof of Theorem \ref{thm:efficiency}]

    Denote an orthonormal basis of $\hat{\mathcal{U}}_m^{ARG}$ by $\hat{U}_m^{ARG}$. By the definition of the principal angles between subspaces, $\theta_k(\hat{\mathcal{U}}_m,\mathcal{U}_m) =\arccos {\sigma_k(\hat{U}_m^{\top}U_m)}$ and $\theta_k(\hat{\mathcal{U}}_m^{ARG},\mathcal{U}_m) = \arccos {\sigma_k((\hat{U}_m^{ARG})^{\top}U_m)}$. For any unit vector $c \in S^{m-1}$, consider two orthogonal decompositions of $P_{\mathcal{S}}(U_mc)$ as follows:
    \begin{equation*}
        \norm{P_{\mathcal{S}}(U_mc)}^2
        = \norm{P_{\hat{\mathcal{U}}^{ARG}_m}(U_mc)}^2 + \norm{P_{\mathcal{D}_r}(U_mc)}^2
        = \norm{P_{\hat{\mathcal{U}}_m}(U_mc)}^2 + \norm{P_{\tilde{\mathcal{V}}_r}(U_mc)}^2.
    \end{equation*}
    Since $D_r^{\top}U_mc \xrightarrow{P} 0_r$ by Theorem \ref{thm:orthogonality} and $(D_r^{\top}D_r)^{-1} \xrightarrow{P} (D^G)^{-1}$ by Lemma \ref{lem:well_conditioned_subspaces}, we obtain
    \begin{equation*}
        \norm{P_{\mathcal{D}_r}(U_mc)}^2 = c^{\top}U_m^{\top}D_r(D_r^{\top}D_r)^{-1}D_r^{\top}U_mc \xrightarrow{P} 0.
    \end{equation*}
    The above implies that the probability limit of $\norm{P_{\hat{\mathcal{U}}^{ARG}_m}(U_mc)}^2 - \norm{P_{\hat{\mathcal{U}}_m}(U_mc)}^2$ is equivalent to that of $\norm{P_{\tilde{\mathcal{V}}_r}(U_mc)}^2$. Using Lemmas \ref{lem:asymp_properties_PC} and \ref{lem:asymp_properties_reference_directions}, we obtain the probability limit of $\tilde{v}_i^{\top}U_mc$ as follows:
    \begin{align*}
        \tilde{v}_i^{\top}U_mc 
        &= v_i^{\top}U_mc - \sum_{j=1}^m (\hat{u}_j^{\top}v_i)(\hat{u}_j^{\top}U_mc) \\
        & \xrightarrow{P} c^{\top}a_i^{vec} - \sum_{j=1}^m \frac{\phi_j(\Omega)}{\phi_j(\Omega) + \tau^2} \left( \sum_{k=1}^m a_{ik}v_{jk}(\Omega)\right) \left(\sum_{l=1}^m c_lv_{jl}(\Omega)\right)\\
        & = (a_i^{vec})^{\top}V(\Omega) \diag\left(\frac{\tau^2}{\phi_j(\Omega) + \tau^2}\right)V(\Omega)^{\top}c.
    \end{align*}
    Therefore, applying Lemma \ref{lem:well_conditioned_subspaces}, the probability limit of $\norm{P_{\tilde{\mathcal{V}}_r}(U_mc)}^2$ becomes
    \begin{align} \label{eq:V_tilde_asymptotics}
        \norm{P_{\tilde{\mathcal{V}}_r}(U_mc)}^2 &= c^{\top}U_m^{\top}\tilde{V}_r(\tilde{V}_r^{\top}\tilde{V}_r)^{-1}\tilde{V}_r^{\top}U_mc \notag \\
        &\xrightarrow{P} \norm{(\tilde{V}^G)^{-\frac{1}{2}}A^{\top}V(\Omega)\diag\left(\frac{\tau^2}{\phi_j(\Omega) + \tau^2}\right)V(\Omega)^{\top}c}^2.
    \end{align}
    
    (i) If $A$ is a full-rank matrix, $(\tilde{V}^G)^{-\frac{1}{2}}A^{\top}V(\Omega)\diag\left(\frac{\tau^2}{\phi_j(\Omega) + \tau^2}\right)V(\Omega)^{\top}$ in \eqref{eq:V_tilde_asymptotics} is also full-rank, which implies:
    \begin{align*}
        &\inf_{c \in S^{m-1}}\norm{(\tilde{V}^G)^{-\frac{1}{2}}A^{\top}V(\Omega)\diag\left(\frac{\tau^2}{\phi_j(\Omega) + \tau^2}\right)V(\Omega)^{\top}c}^2\\
        = &\sigma_{min}\big((\tilde{V}^G)^{-\frac{1}{2}}A^{\top}V(\Omega)\diag\left(\frac{\tau^2}{\phi_j(\Omega) + \tau^2}\right)V(\Omega)^{\top}\big)^2>0.
    \end{align*}
    Since $\norm{(\hat{U}_m^{ARG})^{\top}U_m}_2 \leq \norm{\hat{U}_m^{ARG}}_2 \cdot \norm{U_m}_2 = 1$ and $\norm{(\hat{U}_m)^{\top}U_m}_2 \leq \norm{\hat{U}_m}_2 \cdot \norm{U_m}_2 = 1$, applying Lemma \ref{lem:singular_value_asymptotics}(i) and taking the arccosine, we conclude that as $p \rightarrow \infty$, $\mathbb{P}\big(\theta_{k}(\hat{\mathcal{U}}_m^{ARG}, \mathcal{U}_m) < \theta_{k}(\hat{\mathcal{U}}_m, \mathcal{U}_m)\big) \rightarrow 1$ for $k=1,\dots,m$.

    (ii) If $A$ is neither a full-rank matrix nor a zero matrix, the probability limit of \eqref{eq:V_tilde_asymptotics} is strictly positive almost surely, but its infimum over $S^{m-1}$ is zero due to rank deficiency. Applying Lemma \ref{lem:singular_value_asymptotics}(ii) and taking the arccosine, we conclude that as $p \rightarrow \infty$, $\mathbb{P}\big(\theta_{k}(\hat{\mathcal{U}}_m^{ARG}, \mathcal{U}_m) < \theta_{k}(\hat{\mathcal{U}}_m, \mathcal{U}_m) + \eps\big) \rightarrow 1$ for any $\eps>0$ and $k=1,\dots,m$.
    
    (iii) If $A$ is a zero matrix, then the probability limit of \eqref{eq:V_tilde_asymptotics} is zero. Applying Lemma \ref{lem:singular_value_asymptotics}(iii) and taking the arccosine, we conclude that as $p \rightarrow \infty$, $\theta_{k}(\hat{\mathcal{U}}_m^{ARG}, \mathcal{U}_m) - \theta_{k}(\hat{\mathcal{U}}_m, \mathcal{U}_m) \xrightarrow{P} 0$ for $k=1,\dots,m$.
\end{proof}

Although we have shown that the ARG estimator $\hat{\mathcal{U}}_m^{ARG}$ asymptotically outperforms the naive estimator $\hat{\mathcal{U}}_m$ in terms of principal angles, further analytical progress is limited by the intractability of their explicit characterization. While the probability limit of $(\hat{U}_m^{ARG})^{\top}U_m$ can be obtained, its singular values---governing the principal angles---lack a closed-form expression. Nevertheless, our results offer a rigorous asymptotic comparison framework, confirming the superiority of the ARG estimator in PC subspace estimation.

\section{Numerical studies}

\subsection{A simulation experiment}

We numerically demonstrate that the ARG estimator $\hat{\mathcal{U}}_m^{ARG}$ provides a more accurate estimate of the true PC subspace $\mathcal{U}_m$ than the naive sample PC subspace $\hat{\mathcal{U}}_m$ by presenting two toy examples: (i) a simple case of a single-spike, single-reference direction and (ii) a more general case of $m=2$ spikes and $r=2$ reference directions under Gaussian distributions. The following $p$-dimensional orthonormal vectors are used in defining the true PC and reference directions:

\begin{equation*}
    \begin{split}
    e_1 = \frac{1}{\sqrt{p}}\begin{pmatrix}
        1_{p/4}\\1_{p/4}\\1_{p/4}\\1_{p/4}
    \end{pmatrix},
    e_2 = \frac{1}{\sqrt{p}}\begin{pmatrix}
        1_{p/4}\\1_{p/4}\\-1_{p/4}\\-1_{p/4}
    \end{pmatrix},
    e_3 = \frac{1}{\sqrt{p}}\begin{pmatrix}
        1_{p/4}\\-1_{p/4}\\-1_{p/4}\\1_{p/4}
    \end{pmatrix},
    e_4 = \frac{1}{\sqrt{p}}\begin{pmatrix}
        1_{p/4}\\-1_{p/4}\\1_{p/4}\\-1_{p/4}
    \end{pmatrix}.
    \end{split}
\end{equation*}

First, we report the performance of the ARG estimator under a simple case of a single-spike, single-reference direction. The performance of the estimator is evaluated by comparing $\operatorname{Angle}(\hat{u}_1^{ARG},u_1)$ and $\operatorname{Angle}(\hat{u}_1,u_1)$. Data are generated from a multivariate Gaussian distribution with mean $\mu = 0_p$ and covariance $\Sigma = p e_1 e_1^{\top} + 40 I_p$, which has eigenvalues $(p+40,40,\dots,40)$ and the first eigenvector $e_1$. This setup satisfies Assumption \ref{assum:spiked_covariance} with $\sigma_1^2 = 1$ and $\tau^2 = 40$. For the reference direction, we set $v_1 = a_1 e_1 + \sqrt{1-a_1^2} e_2$ for several values of $a_1^2=\{0,1/4,1/2,3/4,1\}$ ($a_1>0$), ensuring that Assumption \ref{assum:reference_directions_ratios} holds. The sample size is fixed at $n=40$, while the dimension $p$ varies from 100 to 2000. For each combination of $(p,a_1^2)$, the simulation is repeated 100 times.

\begin{table}[t]
    \centering
    \begin{tabular}{c|cccccc}
         \diagbox{$p$}{$a_1^2$} & naive & 0 & 1/4 & 1/2 & 3/4 & 1 \\
         \hline
         \hline
         100  & \makecell{1.0170\\(0.1869)} & \makecell{1.0195\\(0.1867)} & \makecell{0.9649\\(0.1952)} & \makecell{0.8974\\(0.2097)} & \makecell{0.8103\\(0.2361)} & \makecell{0.6841\\(0.2917)} \\
         \hline
         200  & \makecell{0.9157\\(0.1392)} & \makecell{0.9174\\(0.1392)} & \makecell{0.8458\\(0.1359)} & \makecell{0.7548\\(0.1372)} & \makecell{0.6299\\(0.1510)} & \makecell{0.4104\\(0.2142)} \\
         \hline
         500  & \makecell{0.8481\\(0.0885)} & \makecell{0.8491\\(0.0887)} & \makecell{0.7678\\(0.0722)} & \makecell{0.6622\\(0.0581)} & \makecell{0.5096\\(0.0516)} & \makecell{0.1824\\(0.0979)} \\
         \hline
         1000 & \makecell{0.8213\\(0.0737)} & \makecell{0.8221\\(0.0738)} & \makecell{0.7420\\(0.0559)} & \makecell{0.6362\\(0.0384)} & \makecell{0.4793\\(0.0236)} & \makecell{0.0886\\(0.0536)} \\
         \hline
         2000 & \makecell{0.8104\\(0.0651)} & \makecell{0.8108\\(0.0652)} & \makecell{0.7322\\(0.0476)} & \makecell{0.6276\\(0.0306)} & \makecell{0.4708\\(0.0149)} & \makecell{0.0473\\(0.0324)} \\
    \end{tabular}
    
    \caption{Simulation results for the single-spike, single-reference case under a multivariate Gaussian distribution. Values outside parentheses represent the mean angles $\operatorname{Angle}(\hat{u}_1^{ARG}, u_1)$ and $\operatorname{Angle}(\hat{u}_1, u_1)$, and values inside parentheses are their standard deviations, all expressed in radians.}

\label{tab:spike1_simulation}
\end{table}

The simulation results for the single-spike, single-reference case are presented in Table \ref{tab:spike1_simulation}. The first column reports the results for the naive estimator $\hat{u}_1$, and the remaining columns show those for the ARG estimators at different values of $a_1^2$. Each row corresponds to a different dimension $p$.

A clear trend emerges: as $a_1^2$ increases, the ARG estimator $\hat{u}_1^{ARG}$ achieves smaller angles with the true PC direction, indicating improved performance. When $a_1^2=1$, where the reference direction is perfectly aligned with $u_1$, the angle appears to converge to zero as $p \rightarrow \infty$, which is consistent with Theorem \ref{thm:u1_asymptotic}. At the other extreme of $a_1^2=0$, for which the reference direction does not contain useful information, the ARG estimator exhibits a slight performance decrease compared to the naive estimator. While Theorem \ref{thm:u1_asymptotic} suggests that $\hat{u}_1^{ARG}$ and $\hat{u}_1$ should behave identically in the limit, the observed degradation suggests that incorporating prior information, when the reference direction is useless, can introduce a small inefficiency at finite $p$. However, this decrease is negligible, confirming that the ARG estimator maintains its effectiveness even when the reference direction is uninformative. More importantly, as long as the reference direction contains any nontrivial information, the ARG estimator consistently outperforms the naive estimator, demonstrating its practical advantage in PC subspace estimation for HDLSS data.

Next, we report the performance of the ARG estimator under a more general case of two-spike, two-reference direction. The performance of the estimator is evaluated by comparing the principal angles $\theta_k(\hat{\mathcal{U}}_m^{ARG}, \mathcal{U}_m)$ and $\theta_k(\hat{\mathcal{U}}_m, \mathcal{U}_m)$, for  $k =1,2$ and $m = 2$. Data are generated from a multivariate Gaussian distribution with mean $\mu = 0_p$ and covariance $\Sigma = 2p e_1 e_1^{\top} + pe_2e_2^{\top} + 40 I_p$, which has eigenvalues $(2p+40,p+40,40,\dots,40)$ and the first and second eigenvectors $e_1,e_2$. This setup satisfies Assumption \ref{assum:spiked_covariance} with $\sigma_1^2 = 2$, $\sigma_2^2=1$, and $\tau^2 = 40$. For reference directions, we fix $v_1 = e_1/2 + e_2/2 + e_3/2 + e_4/2$ and $v_2 = e_1/\sqrt{2} - e_3/\sqrt{2}$, ensuring that Assumption \ref{assum:reference_directions_ratios} holds with $a_{11} = a_{12} = 1/2$, $a_{21} = 1/\sqrt{2}$ and $a_{22} = 0$. The sample size is fixed at $n=40$, while the dimension $p$ varies from 100 to 2000. For each $p$, the simulation is repeated 100 times.

\begin{table}[t]
    \centering
    \begin{tabular}{c|cc|cc}
         $p$ & $\theta_1(\hat{\mathcal{U}}_m^{ARG},\mathcal{U}_m)$ & $\theta_1(\hat{\mathcal{U}}_m,\mathcal{U}_m)$ & $\theta_2(\hat{\mathcal{U}}_m^{ARG},\mathcal{U}_m)$ & $\theta_2(\hat{\mathcal{U}}_m,\mathcal{U}_m)$ \\
         \hline
         \hline
         100  & \makecell{0.4694\\(0.0995)} & \makecell{0.6873\\(0.1024)} & \makecell{1.0248\\(0.2208)} & \makecell{1.0805\\(0.2083)} \\
         \hline
         200  & \makecell{0.4067\\(0.0621)} & \makecell{0.6466\\(0.0715)} & \makecell{0.8754\\(0.1769)} & \makecell{0.9586\\(0.1724)} \\
         \hline
         500  & \makecell{0.3632\\(0.0219)} & \makecell{0.6348\\(0.0616)} & \makecell{0.7645\\(0.0641)} & \makecell{0.8514\\(0.0783)} \\
         \hline
         1000 & \makecell{0.3518\\(0.0131)} & \makecell{0.6276\\(0.0549)} & \makecell{0.7506\\(0.0668)} & \makecell{0.8374\\(0.0870)} \\
         \hline
         2000 & \makecell{0.3489\\(0.0098)} & \makecell{0.6202\\(0.0479)} & \makecell{0.7357\\(0.0521)} & \makecell{0.8201\\(0.0628)} \\
         \hline
    \end{tabular}

    \caption{Simulation results for the two-spike, two-reference case under a multivariate Gaussian distribution. Values outside parentheses represent the mean principal angles $\theta_k(\hat{\mathcal{U}}_m^{ARG}, \mathcal{U}_m)$ and $\theta_k(\hat{\mathcal{U}}_m, \mathcal{U}_m)$, and values inside parentheses are their standard deviations, all expressed in radians.}
    \label{tab:general_simulation}
\end{table}

The simulation results for the two-spike, two-reference direction case are presented in Table \ref{tab:general_simulation}. Each row corresponds to a different dimension $p$, while the first two columns report the first principal angles between the estimators and the true PC subspace $\mathcal{U}_m$, and the latter two report the second principal angles. The results consistently show a reduction in all principal angles when using the ARG estimator, demonstrating improved accuracy in subspace estimation.

The above experiments are replicated under a multivariate $t$-distribution, for which the $\rho$-mixing condition of Assumption \ref{assum:rho_mixing} is violated. The ARG estimator  remains to outperform in such situations even though there is no theoretical guarantee; see the \nameref{sec:supporting_information} for details.

\subsection{Real data analysis}

We demonstrate the application of the ARG estimator $\hat{\mathcal{U}}_m^{ARG}$ using daily log-return data from the NASDAQ stock market. Daily adjusted closing prices for all NASDAQ-listed stocks throughout 2024 are collected, from which log-returns are computed. To ensure data completeness, we retain only those stocks with no missing values throughout the entire year, resulting in $p = 3748$ variables. To reflect the HDLSS setting, we focus on the $n = 20$ trading days in December 2024, treating each day as an observation. Assuming the number of spikes $m = 2$, we compare (i) standard PCA and (ii) ARG-PCA, which takes advantage of the improved performance of the ARG estimator. A detailed description of the ARG-PCA algorithm is provided in the \nameref{sec:supporting_information}. As reference directions, we use the normalized one vector and the mean vector of daily log-returns over the full year 2024.

\begin{figure}
    \centering
    \includegraphics[width=1\linewidth]{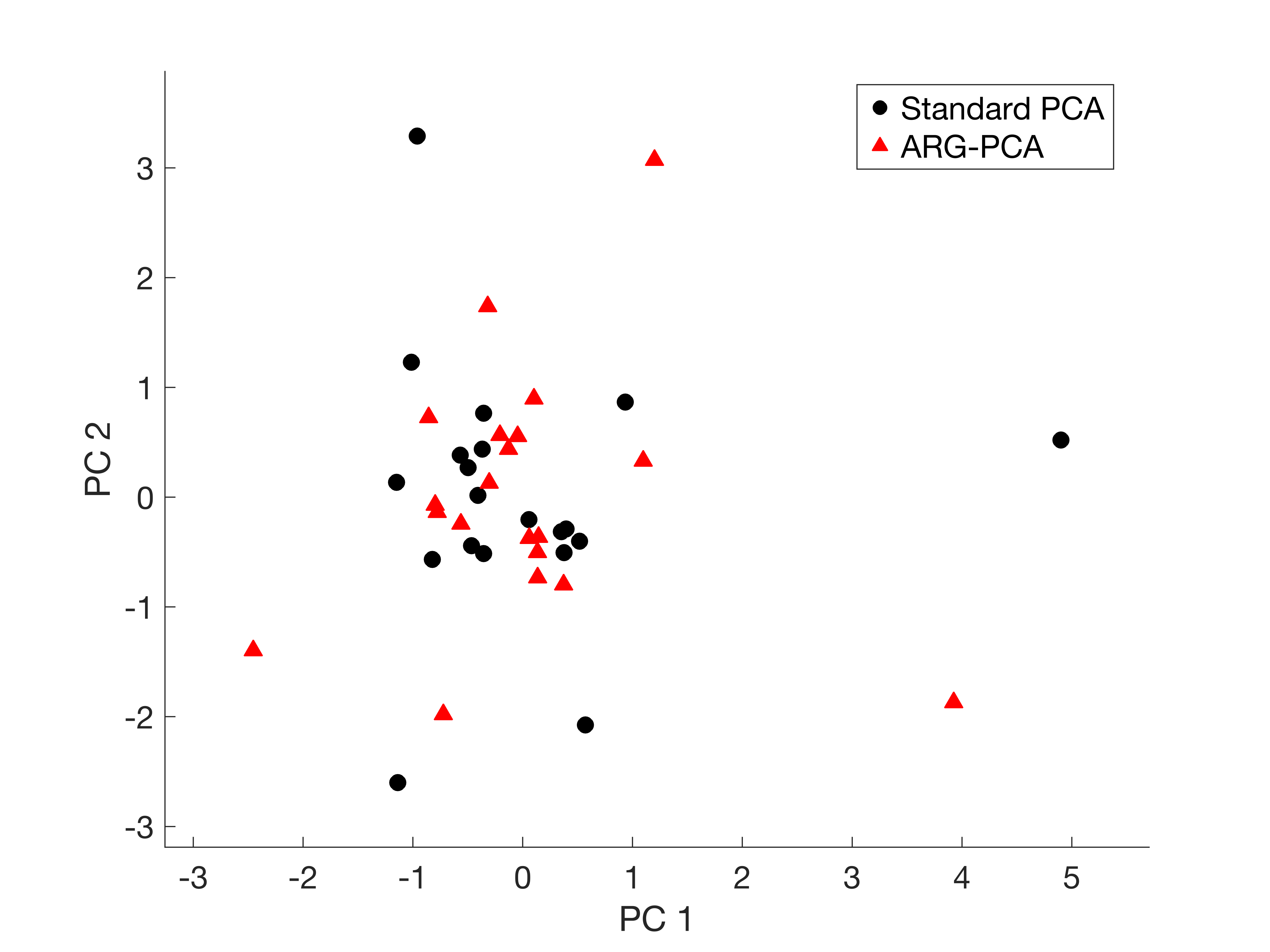}
    \caption{Score plot of the first two estimated PCs. Standard PCA scores are shown as black circles ($\bullet$), and ARG-PCA scores are shown as red triangles (\textcolor{red}{$\blacktriangle$}).}
    \label{fig:real_data}
\end{figure}

Figure~\ref{fig:real_data} presents the score plot based on the first two estimated PCs from both the standard PCA and the ARG-PCA. Although the true PC subspace $\mathcal{U}_m$ is unknown and direct evaluation is not possible, the two sets of PC scores show visible differences in orientation, with the ARG PC scores appearing as a rotated version of the standard PC scores. To quantify this difference, we compute the first and second principal angles between the ARG estimator $\hat{\mathcal{U}}_m^{ARG}$ and the sample PC subspace $\hat{\mathcal{U}}_m$:
\begin{equation*}
\theta_1(\hat{\mathcal{U}}_m^{ARG},\hat{\mathcal{U}}_m) = 0.1317, \quad \theta_2(\hat{\mathcal{U}}_m^{ARG},\hat{\mathcal{U}}_m) = 0.4544.
\end{equation*}
These nonzero angles confirm that $\hat{\mathcal{U}}_m^{ARG}$ differs meaningfully from $\hat{\mathcal{U}}_m$, potentially capturing a subspace that better reflects the underlying data structure.

\section{Discussion}
In high-dimensional covariance estimation, a common structural assumption is the low-rank plus sparse decomposition, as adopted in methods such as POET \citep{Fan2013} and more recent approaches based on nuclear norm plus $\ell_1$ norm penalization \citep{Chandrasekaran2011,Chandrasekaran2012,Farne2020,Farne2024}. While we focus on leveraging auxiliary information, these methods focus on exploiting structural assumptions on the covariance matrix. Thus, our approach may serve as a complementary tool for recovering the low-rank  structure in covariance.

\bibliographystyle{apalike}
\bibliography{library}

\begin{thebibliography}{}

\bibitem[Ahn et~al., 2007]{Ahn2007}
Ahn, J., Marron, J.~S., Muller, K.~M., and Chi, Y.-Y. (2007).
\newblock The high-dimension, low-sample-size geometric representation holds under mild conditions.
\newblock {\em Biometrika}, 94(3):760--766.

\bibitem[Bai and Ng, 2002]{Bai2002}
Bai, J. and Ng, S. (2002).
\newblock Determining the number of factors in approximate factor models.
\newblock {\em Econometrica}, 70(1):191--221.

\bibitem[Bartlett et~al., 2020]{Bartlett2020}
Bartlett, P.~L., Long, P.~M., Lugosi, G., and Tsigler, A. (2020).
\newblock Benign overfitting in linear regression.
\newblock {\em Proceedings of the National Academy of Sciences}, 117(48):30063--30070.

\bibitem[Bradley, 2005]{Bradley2005}
Bradley, R.~C. (2005).
\newblock Basic properties of strong mixing conditions. a survey and some open questions.
\newblock {\em Probability Surveys}, 2:107 -- 144.

\bibitem[Chamberlain and Rothschild, 1983]{Chamberlain1983}
Chamberlain, G. and Rothschild, M. (1983).
\newblock Arbitrage, factor structure, and mean-variance analysis on large asset markets.
\newblock {\em Econometrica}, 51(5):1281--1304.

\bibitem[Chandrasekaran et~al., 2012]{Chandrasekaran2012}
Chandrasekaran, V., Parrilo, P.~A., and Willsky, A.~S. (2012).
\newblock Latent variable graphical model selection via convex optimization.
\newblock {\em The Annals of Statistics}, 40(4):1935--1967.

\bibitem[Chandrasekaran et~al., 2011]{Chandrasekaran2011}
Chandrasekaran, V., Sanghavi, S., Parrilo, P.~A., and Willsky, A.~S. (2011).
\newblock Rank-sparsity incoherence for matrix decomposition.
\newblock {\em SIAM Journal on Optimization}, 21(2):572--596.

\bibitem[Chang et~al., 2021]{Chang2021}
Chang, W., Ahn, J., and Jung, S. (2021).
\newblock {Double data piling leads to perfect classification}.
\newblock {\em Electronic Journal of Statistics}, 15(2):6382 -- 6428.

\bibitem[Donoho et~al., 2018]{Donoho2018}
Donoho, D., Gavish, M., and Johnstone, I. (2018).
\newblock {Optimal shrinkage of eigenvalues in the spiked covariance model}.
\newblock {\em The Annals of Statistics}, 46(4):1742 -- 1778.

\bibitem[Fama and French, 1993]{Fama1993}
Fama, E.~F. and French, K.~R. (1993).
\newblock Common risk factors in the returns on stocks and bonds.
\newblock {\em Journal of Financial Economics}, 33(1):3--56.

\bibitem[Fama and French, 2015]{Fama2015}
Fama, E.~F. and French, K.~R. (2015).
\newblock A five-factor asset pricing model.
\newblock {\em Journal of Financial Economics}, 116(1):1--22.

\bibitem[Fan et~al., 2011]{Fan2011}
Fan, J., Liao, Y., and Mincheva, M. (2011).
\newblock {High-dimensional covariance matrix estimation in approximate factor models}.
\newblock {\em The Annals of Statistics}, 39(6):3320 -- 3356.

\bibitem[Fan et~al., 2013]{Fan2013}
Fan, J., Liao, Y., and Mincheva, M. (2013).
\newblock Large covariance estimation by thresholding principal orthogonal complements.
\newblock {\em Journal of the Royal Statistical Society. Series B (Statistical Methodology)}, 75(4):603--680.

\bibitem[Farnè and Montanari, 2020]{Farne2020}
Farnè, M. and Montanari, A. (2020).
\newblock A large covariance matrix estimator under intermediate spikiness regimes.
\newblock {\em Journal of Multivariate Analysis}, 176:104577.

\bibitem[Farnè and Montanari, 2024]{Farne2024}
Farnè, M. and Montanari, A. (2024).
\newblock Large factor model estimation by nuclear norm plus $\ell_1$ norm penalization.
\newblock {\em Journal of Multivariate Analysis}, 199:105244.

\bibitem[Goldberg and Kercheval, 2023]{Goldberg2023}
Goldberg, L.~R. and Kercheval, A.~N. (2023).
\newblock James–stein for the leading eigenvector.
\newblock {\em Proceedings of the National Academy of Sciences}, 120(2):e2207046120.

\bibitem[Goldberg et~al., 2022]{Goldberg2022}
Goldberg, L.~R., Papanicolaou, A., and Shkolnik, A. (2022).
\newblock The dispersion bias.
\newblock {\em SIAM Journal on Financial Mathematics}, 13(2):521--550.

\bibitem[Goldberg et~al., 2020]{Goldberg2020}
Goldberg, L.~R., Papanicolaou, A., Shkolnik, A., and Ulucam, S. (2020).
\newblock Better betas.
\newblock {\em The Journal of Portfolio Management}, 47(1):119 -- 136.

\bibitem[Gurdogan and Kercheval, 2022]{Gurdogan2022}
Gurdogan, H. and Kercheval, A. (2022).
\newblock Multiple anchor point shrinkage for the sample covariance matrix.
\newblock {\em SIAM Journal on Financial Mathematics}, 13(3):1112--1143.

\bibitem[Gurdogan and Shkolnik, 2024]{Gurdogan2024}
Gurdogan, H. and Shkolnik, A. (2024).
\newblock The quadratic optimization bias of large covariance matrices.
\newblock arXiv no. 2410.03053.

\bibitem[Hall et~al., 2005]{Hall2005}
Hall, P., Marron, J.~S., and Neeman, A. (2005).
\newblock Geometric representation of high dimension, low sample size data.
\newblock {\em Journal of the Royal Statistical Society. Series B (Statistical Methodology)}, 67(3):427--444.

\bibitem[Johnstone, 2001]{Johnstone2001}
Johnstone, I.~M. (2001).
\newblock On the distribution of the largest eigenvalue in principal components analysis.
\newblock {\em The Annals of Statistics}, 29(2):295--327.

\bibitem[Jung and Marron, 2009]{Jung2009}
Jung, S. and Marron, J.~S. (2009).
\newblock {PCA consistency in high dimension, low sample size context}.
\newblock {\em The Annals of Statistics}, 37(6B):4104 -- 4130.

\bibitem[Jung et~al., 2012]{Jung2012}
Jung, S., Sen, A., and Marron, J. (2012).
\newblock {Boundary behavior in High Dimension, Low Sample Size asymptotics of PCA}.
\newblock {\em Journal of Multivariate Analysis}, 109:190--203.

\bibitem[Kanehisa and Goto, 2000]{Kanehisa2000}
Kanehisa, M. and Goto, S. (2000).
\newblock {KEGG}: {Kyoto} {Encyclopedia} of {Genes} and {Genomes}.
\newblock {\em Nucleic Acids Research}, 28(1):27--30.

\bibitem[Ke et~al., 2023]{Ke2023}
Ke, Z.~T., Ma, Y., and Lin, X. (2023).
\newblock Estimation of the number of spiked eigenvalues in a covariance matrix by bulk eigenvalue matching analysis.
\newblock {\em Journal of the American Statistical Association}, 118(541):374--392.

\bibitem[Kim et~al., 2024]{Kim2024}
Kim, T., Chang, W., Ahn, J., and Jung, S. (2024).
\newblock Double data piling: a high-dimensional solution for asymptotically perfect multi-category classification.
\newblock {\em Journal of the Korean Statistical Society}, 53:704--737.

\bibitem[Kobak et~al., 2020]{Kobak2020}
Kobak, D., Lomond, J., and Sanchez, B. (2020).
\newblock The optimal ridge penalty for real-world high-dimensional data can be zero or negative due to the implicit ridge regularization.
\newblock {\em Journal of Machine Learning Research}, 21(1).

\bibitem[Kolmogorov and Rozanov, 1960]{Kolmogorov1960}
Kolmogorov, A.~N. and Rozanov, Y.~A. (1960).
\newblock On strong mixing conditions for stationary gaussian processes.
\newblock {\em Theory of Probability \& Its Applications}, 5(2):204--208.

\bibitem[Lintner, 1965]{Lintner1965}
Lintner, J. (1965).
\newblock The valuation of risk assets and the selection of risky investments in stock portfolios and capital budgets.
\newblock {\em The Review of Economics and Statistics}, 47(1):13--37.

\bibitem[Sharpe, 1964]{Sharpe1964}
Sharpe, W.~F. (1964).
\newblock Capital asset prices: A theory of market equilibrium under conditions of risk.
\newblock {\em The Journal of Finance}, 19(3):425--442.

\bibitem[Shkolnik, 2022]{Shkolnik2022}
Shkolnik, A. (2022).
\newblock {James–Stein estimation of the first principal component}.
\newblock {\em Stat}, 11(1):e419.

\bibitem[Shkolnik et~al., 2025]{Shkolnik2025}
Shkolnik, A., Kercheval, A., Gurdogan, H., Goldberg, L.~R., and Bar, H. (2025).
\newblock Portfolio selection revisited.
\newblock {\em Annals of Operations Research}, 346(1):137--155.

\bibitem[Yata and Aoshima, 2012]{Yata2012}
Yata, K. and Aoshima, M. (2012).
\newblock Effective {PCA} for high-dimension, low-sample-size data with noise reduction via geometric representations.
\newblock {\em Journal of Multivariate Analysis}, 105(1):193--215.

\end{thebibliography}

\clearpage
\phantomsection
\section*{Supporting Information} \label{sec:supporting_information}

The \nameref{sec:supporting_information} includes omitted proofs, additional simulation results for the multivariate $t$-distribution, a detailed description of the ARG-PCA algorithm, and the MATLAB code used in this study.

\subsection*{Appendix S1: Proof of Lemma 6}

\begin{proof} [Proof of Lemma \ref{lem:singular_value_asymptotics}]

    (i) Consider the following min-max theorem for singular values: the $k$th largest singular value of $m \times m$ matrix $M$ can be represented as
    \begin{equation*}
        \sigma_k (M) = \min_{\dim{\mathcal{T}} = m-k+1} \max_{c \in \mathcal{T} \cap S^{m-1}} \norm{Mc}.
    \end{equation*}
    Denote $f_M(\mathcal{T}) = \max_{c \in \mathcal{T} \cap S^{m-1}} \norm{Mc}$, $\mathcal{T}_{M,k} = \argmin_{\dim{\mathcal{T}}=m-k+1}f_M(\mathcal{T})$, and $c_{M,\mathcal{T}} = \argmax_{c \in \mathcal{T} \cap S^{m-1}}\norm{Mc}$. Then, we can lower bound $\sigma_k(A_p) - \sigma_k(B_p)$ as follows:

    \begin{align*}
        \sigma_k(A_p) - \sigma_k(B_p) & = f_{A_p}(\mathcal{T}_{A_p,k}) - f_{B_p}(\mathcal{T}_{B_p,k}) \notag \\
        & \geq f_{A_p}(\mathcal{T}_{A_p,k}) - f_{B_p}(\mathcal{T}_{A_p,k}) \notag \\
        & = \norm{A_pc_{A_p,\mathcal{T}_{A_p,k}}} - \norm{B_pc_{B_p,\mathcal{T}_{A_p,k}}} \notag \\
        & \geq \norm{A_pc_{B_p,\mathcal{T}_{A_p,k}}} - \norm{B_pc_{B_p,\mathcal{T}_{A_p,k}}} \notag \\
        & \geq \inf_{c \in S^{m-1}} (\norm{A_pc} - \norm{B_pc}).
    \end{align*}

    Fix $\eps>0$ and $c \in S^{m-1}$. Denote an $\eps$-cover of $S^{m-1}$ as $\{c_1,\dots,c_{N(\eps)}\}$. Then, there exists $i^* \in \{1,\dots,N(\eps)\}$ such that $\norm{c-c_{i^*}}<\eps$. Since $\norm{A_p}_2 \leq 1$ and $\norm{B_p}_2 \leq 1$ for all $p$, we get:
    
    \begin{align} \label{eq:norm_diff_positive_decomposition}
        \norm{A_pc} - \norm{B_pc} &= (\norm{A_pc} - \norm{A_pc_{i^*}}) + (\norm{A_pc_{i^*}} - \norm{B_pc_{i^*}}) + (\norm{B_pc_{i^*}} - \norm{B_pc}) \notag \\
        & \geq -\norm{A_p(c-c_{i^*})} + \min_{1 \leq i \leq N(\eps)} (\norm{A_pc_i} - \norm{B_pc_i}) - \norm{B_p(c-c_{i^*})} \notag \\
        & \geq -\norm{A_p}_2 \norm{c-c_{i^*}} + \min_{1 \leq i \leq N(\eps)} (\norm{A_pc_i} - \norm{B_pc_i}) - \norm{B_p}_2\norm{c-c_{i^*}} \notag \\
        & \geq \min_{1 \leq i \leq N(\eps)} (\norm{A_pc_i} - \norm{B_pc_i}) - 2\eps.
    \end{align}
    As $\norm{A_pc_i} - \norm{B_pc_i} \xrightarrow{P}X_{c_i}$ and $X_{c_i}>0$ almost surely, $\mathbb{P}(\norm{A_pc_i} - \norm{B_pc_i} > \frac{1}{2}X_{c_i}) \rightarrow 1$ as $p \rightarrow \infty$ for $i=1,\dots,N(\eps)$. Define the event $E_p = \cap_{i=1}^{N(\eps)}\{\norm{A_pc_i} - \norm{B_pc_i}>\frac{1}{2}X_{c_i}\}$. Since $N(\eps)$ is independent of $p$, $\sum_{i=1}^{N(\eps)} \mathbb{P}(\norm{A_pc_i} - \norm{B_pc_i} \leq \frac{1}{2}X_{c_i}) \rightarrow 0$ so that $\mathbb{P}(E_p) \rightarrow 1$ as $p \rightarrow \infty$. On the event $E_p$,
    
    \begin{equation} \label{eq:event_Ep}
        \min_{1 \leq i \leq N(\eps)}(\norm{A_pc_i} - \norm{B_pc_i}) > \frac{1}{2} \min_{1 \leq i \leq N(\eps)}X_{c_i} \geq \frac{1}{2} \inf_{c \in S^{m-1}}X_c
    \end{equation}
    holds. Combining (\ref{eq:norm_diff_positive_decomposition}) and (\ref{eq:event_Ep}), we obtain:

    \begin{align} \label{eq:norm_diff_positive}
        & \mathbb{P}(\inf_{c \in S^{m-1}}(\norm{A_pc} - \norm{B_pc}) \leq 0) \notag \\
        = & \mathbb{P}(\{\inf_{c \in S^{m-1}}(\norm{A_pc} - \norm{B_pc}) \leq 0\} \cap E_p) + \mathbb{P}(\{\inf_{c \in S^{m-1}}(\norm{A_pc} - \norm{B_pc}) \leq 0\} \cap E_p^c) \notag \\
        \leq & \mathbb{P}(\frac{1}{2}\inf_{c \in S^{m-1}}X_c \leq 2\eps) + \mathbb{P}(E_p^c).
    \end{align}
    Since $\inf_{c \in S^{m-1}}X_c > 0$ holds almost surely, taking $\eps \rightarrow 0$, the first term of the upper bound of (\ref{eq:norm_diff_positive}) converges to zero; the second term converges to zero as $p \rightarrow \infty$. Therefore, we conclude that as $p \rightarrow \infty$, for $k=1,\dots,m$,
    
    \begin{equation*}
        \mathbb{P}(\sigma_k(A_p) > \sigma_k(B_p)) \geq \mathbb{P}(\inf_{c \in S^{m-1}}(\norm{A_pc} - \norm{B_pc})>0) \rightarrow 1.
    \end{equation*}

    (ii) Repeating the same argument as (i) with $(\eps/3)$-cover of $S^{m-1}$, we obtain:

    \begin{equation} \label{eq:norm_diff_non_negative}
        \mathbb{P}(\inf_{c \in S^{m-1}}(\norm{A_pc} - \norm{B_pc}) \leq- \eps) \leq \mathbb{P}(\frac{1}{2}\inf_{c \in S^{m-1}}X_c \leq -\frac{1}{3}\eps) + \mathbb{P}(E_p^c).
    \end{equation}
    Since $\inf_{c \in S^{m-1}}X_c=0$, the first term of the upper bound of (\ref{eq:norm_diff_non_negative}) equals zero; the second term converges to zero as $p \rightarrow \infty$. Therefore, we conclude that as $p \rightarrow \infty$, for any $\eps>0$ and $k=1,\dots,m$,

    \begin{equation*}
        \mathbb{P}(\sigma_k(A_p) > \sigma_k(B_p) - \eps) \geq \mathbb{P}(\inf_{c \in S^{m-1}}(\norm{A_pc} - \norm{B_pc})>-\eps) \rightarrow 1.
    \end{equation*}

    (iii) We first prove that the convergence in probability holds uniformly over $c \in S^{m-1}$, that is,

    \begin{equation} \label{eq:unif_convergence}
        \sup_{c \in S^{m-1}} | \norm{A_pc} - \norm{B_pc}| \xrightarrow{P} 0.
    \end{equation}
    Fix $\eps>0$ and $c \in S^{m-1}$. Denote an $(\eps/3)$-cover of $S^{m-1}$ as $\{c_1,...,c_{N(\eps/3)}\}$. Then, $\exists i^* \in \{1,\dots,N(\eps/3)\}$ such that $\norm{c-c_{i^*}} < \eps/3$. Since $\norm{A_p}_2 \leq 1$ and $\norm{B_p}_2 \leq 1$ for all $p$, we get:

    \begin{align*}
        \abs{\norm{A_pc} - \norm{B_pc}} &\leq \abs{\norm{A_pc} - \norm{A_pc_{i^*}}} + \abs{\norm{A_pc_{i^*}} - \norm{B_pc_{i^*}}} + \abs{\norm{B_pc_{i^*}} - \norm{B_pc}} \notag \\
        & \leq \norm{A_p(c-c_{i^*})} + \max_{1 \leq i \leq N(\eps/3)} \abs{\norm{A_pc_i} - \norm{B_pc_i}} + \norm{B_p(c-c_{i^*})} \notag \\
        & \leq \norm{A_p}_2 \norm{c-c_{i^*}} + \max_{1 \leq i \leq N(\eps/3)} \abs{\norm{A_pc_i} - \norm{B_pc_i}} + \norm{B_p}_2\norm{c-c_{i^*}} \notag \\
        & \leq \max_{1 \leq i \leq N(\eps/3)} \abs{\norm{A_pc_i} - \norm{B_pc_i}} + \frac{2\eps}{3}.
    \end{align*}
    As maximum of finite number of random variables converging to zero in probability also converges to zero in probability, we obtain:

    \begin{equation*}
        \mathbb{P}(\sup_{c \in S^{m-1}}\abs{\norm{A_pc} - \norm{B_pc}} > \eps) \leq \mathbb{P}(\max_{1 \leq i \leq N(\eps/3)}\abs{\norm{A_pc_i} - \norm{B_pc_i}} > \frac{\eps}{3}) \rightarrow 0,
    \end{equation*}
    and therefore the uniform convergence (\ref{eq:unif_convergence}) holds.
    
    For arbitrary subspace $\mathcal{T} \subseteq \Real^m$, we have:

    \begin{align*}
        \abs{f_{A_p}(\mathcal{T}) - f_{B_p}(\mathcal{T})} & = (f_{A_p}(\mathcal{T}) - f_{B_p}(\mathcal{T})) \lor (f_{B_p}(\mathcal{T}) - f_{A_p}(\mathcal{T})) \notag \\
        & = (\norm{A_pc_{A_p,\mathcal{T}}} - \norm{B_pc_{B_p,\mathcal{T}}}) \lor (\norm{B_pc_{B_p,\mathcal{T}}} - \norm{A_pc_{A_p,\mathcal{T}}}) \notag \\
        & \leq (\norm{A_pc_{A_p,\mathcal{T}}} - \norm{B_pc_{A_p,\mathcal{T}}}) \lor (\norm{B_pc_{B_p,\mathcal{T}}} - \norm{A_pc_{B_p,\mathcal{T}}}) \notag \\
        & \leq \sup_{c \in S^{m-1}} \abs{\norm{A_pc} - \norm{B_pc}}.
    \end{align*}
    Next, consider the singular values $\sigma_k(A_p)$ and $\sigma_k(B_p)$. By definition,

    \begin{align*}
        \abs{\sigma_k(A_p) - \sigma_k(B_p)} & = \abs{f_{A_p}(\mathcal{T}_{A_p,k}) - f_{B_p}(\mathcal{T}_{B_p,k})} \notag \\
        & = (f_{A_p}(\mathcal{T}_{A_p,k}) - f_{B_p}(\mathcal{T}_{B_p,k})) \lor (f_{B_p}(\mathcal{T}_{B_p,k}) - f_{A_p}(\mathcal{T}_{A_p,k})) \notag \\
        & \leq (f_{A_p}(\mathcal{T}_{B_p,k}) - f_{B_p}(\mathcal{T}_{B_p,k})) \lor (f_{B_p}(\mathcal{T}_{A_p,k}) - f_{A_p}(\mathcal{T}_{A_p,k})) \notag \\
        & \leq \sup_{\mathcal{T}\subseteq \Real^m} \abs{f_{A_p}(\mathcal{T})-f_{B_p}(\mathcal{T})} \notag \\
        & \leq \sup_{c \in S^{m-1}} \abs{\norm{A_pc} - \norm{B_pc}}.
    \end{align*}
    This bound demonstrates that the difference between singular values is upper bounded by the supremum of difference between norms $\norm{A_pc}$ and $\norm{B_pc}$, which converges to zero in probability. Therefore, we conclude that as $p \rightarrow \infty$, for $k=1,\dots,m$,

    \begin{equation*}
        \sigma_k(A_p) - \sigma_k(B_p) \xrightarrow{P} 0.
    \end{equation*}
\end{proof}

\subsection*{Appendix S2: Simulation results under multivariate $t$-distribution}

In this section, we report the performance of the ARG estimator when the underlying data are drawn from a multivariate $t$-distribution with 5 degrees of freedom. This distribution violates Assumption~\ref{assum:rho_mixing}, as it does not satisfy the $\rho$-mixing condition required for the theoretical guarantees established in the main text. Aside from the distributional assumption, all other simulation conditions remain the same.

\begin{table}[t]
    \centering
    \begin{tabular}{c|cccccc}
        \diagbox{$p$}{$a_1^2$} & naive & 0 & 1/4 & 1/2 & 3/4 & 1 \\
        \hline
        \hline
        100  & \makecell{1.3405\\(0.1622)} & \makecell{1.3410\\(0.1618)} & \makecell{1.3274\\(0.1718)} & \makecell{1.3135\\(0.1839)} & \makecell{1.2988\\(0.1982)} & \makecell{1.2834\\(0.2147)} \\
        \hline
        200  & \makecell{1.3608\\(0.1395)} & \makecell{1.3610\\(0.1393)} & \makecell{1.3478\\(0.1494)} & \makecell{1.3333\\(0.1608)} & \makecell{1.3177\\(0.1741)} & \makecell{1.3000\\(0.1900)} \\
        \hline
        500  & \makecell{1.3778\\(0.1582)} & \makecell{1.3779\\(0.1581)} & \makecell{1.3651\\(0.1695)} & \makecell{1.3510\\(0.1834)} & \makecell{1.3356\\(0.2010)} & \makecell{1.3180\\(0.2241)} \\
        \hline
        1000 & \makecell{1.3516\\(0.1657)} & \makecell{1.3517\\(0.1657)} & \makecell{1.3354\\(0.1795)} & \makecell{1.3180\\(0.1959)} & \makecell{1.2988\\(0.2155)} & \makecell{1.2777\\(0.2391)} \\
        \hline
        2000 & \makecell{1.3597\\(0.1484)} & \makecell{1.3597\\(0.1484)} & \makecell{1.3457\\(0.1594)} & \makecell{1.3304\\(0.1721)} & \makecell{1.3138\\(0.1872)} & \makecell{1.2953\\(0.2056)} \\
    \end{tabular}
    \caption{Simulation results for the single-spike, single-reference direction case under a multivariate $t$-distribution. Values outside parentheses represent the mean angles $\operatorname{Angle}(\hat{u}_1^{ARG}, u_1)$ and $\operatorname{Angle}(\hat{u}_1, u_1)$, and values inside parentheses are their standard deviations, all expressed in radians.}
    \label{tab:spike1_simulation_t}
\end{table}

In the single-spike, single-reference direction setting, the performance of the ARG estimator exhibits minimal improvement as the information quantity increases (i.e., as $a_1^2$ increases). This is in sharp contrast to the Gaussian case, where increase in information quantity leads to substantial performance gains. The lack of $\rho$-mixing appears to hinder the estimator’s ability to effectively leverage the directional information provided by the reference direction. The results are summarized in Table \ref{tab:spike1_simulation_t}.

\begin{table}[t]
    \centering
    \begin{tabular}{c|cc|cc}
        $p$ & $\theta_1(\hat{\mathcal{U}}_m^{ARG},\mathcal{U}_m)$ & $\theta_1(\hat{\mathcal{U}}_m,\mathcal{U}_m)$ & $\theta_2(\hat{\mathcal{U}}_m^{ARG},\mathcal{U}_m)$ & $\theta_2(\hat{\mathcal{U}}_m,\mathcal{U}_m)$ \\
        \hline
        \hline
        100  & \makecell{0.8554\\(0.2137)} & \makecell{0.9493\\(0.1755)} & \makecell{1.3519\\(0.1658)} & \makecell{1.3644\\(0.1533)} \\
        \hline
        200  & \makecell{0.8092\\(0.2168)} & \makecell{0.9275\\(0.1718)} & \makecell{1.3440\\(0.1773)} & \makecell{1.3613\\(0.1594)} \\
        \hline
        500  & \makecell{0.8428\\(0.2338)} & \makecell{0.9519\\(0.1815)} & \makecell{1.3434\\(0.1740)} & \makecell{1.3643\\(0.1557)} \\
        \hline
        1000 & \makecell{0.8187\\(0.2277)} & \makecell{0.9415\\(0.1779)} & \makecell{1.3513\\(0.1636)} & \makecell{1.3711\\(0.1486)} \\
        \hline
        2000 & \makecell{0.7945\\(0.2209)} & \makecell{0.9270\\(0.1628)} & \makecell{1.3491\\(0.1496)} & \makecell{1.3678\\(0.1380)} \\
        \hline
    \end{tabular}
    \caption{Simulation results for the two-spike, two-reference direction case under a multivariate $t$-distribution. Values outside parentheses represent the mean principal angles $\theta_k(\hat{\mathcal{U}}_m^{ARG}, \mathcal{U}_m)$ and $\theta_k(\hat{\mathcal{U}}_m, \mathcal{U}_m)$, and values inside parentheses are their standard deviations, all expressed in radians.}
    \label{tab:general_simulation_t}
\end{table}

In the multi-spike, multi-reference direction setting, we observe modest improvement in the first principal angle, suggesting that the ARG estimator retains some benefit from prior information even in the absence of $\rho$-mixing. However, the second principal angle remains largely unaffected, mirroring the behavior observed in the single-spike case. The results are summarized in Table \ref{tab:general_simulation_t}.

Nevertheless, in all settings considered, the ARG estimator consistently outperforms or matches the naive estimator. These results imply that, in practice, the ARG estimator remains a reasonable and safe choice when prior information about the true PC subspace is available, even under deviations from ideal theoretical conditions such as the $\rho$-mixing assumption.

\subsection*{Appendix S3: ARG-PCA}

We present a new variant of PCA, named \textit{ARG-PCA}, which utilizes the improved asymptotic accuracy of the ARG estimator $\hat{\mathcal{U}}_m^{ARG}$. The detailed procedure for ARG-PCA is presented in Algorithm \ref{alg:ARG-PCA}. We assume that the data were centered in advance.

\begin{algorithm}
\caption{ARG-PCA Algorithm}
\label{alg:ARG-PCA}
\begin{algorithmic}

    \State \textbf{Input:} The centered data matrix $ X_{cen}\in \Real^{p \times n}$, the matrix of reference directions $V_r \in \Real^{p \times r}$, and the number of spikes $m$.
    \State \textbf{Output:} ARG PC variances $\hat{\lambda}_1^{ARG},\dots,\hat{\lambda}_m^{ARG}$, directions $\hat{u}_1^{ARG},\dots,\hat{u}_m^{ARG}$, and scores $\hat{w}_1^{ARG},\dots,\hat{w}_n^{ARG}$.\\
    \State \textbf{Step 1.} Perform PCA on the centered data $X_{cen}$ and compute a non-orthonormal basis of $\hat{\mathcal{U}}_m^{ARG}$ using:
    \[(S_m - \tilde{\lambda}I_p)(I_p - P_{\mathcal{V}_r})\hat{U}_m.\]
    \State \textbf{Step 2.} Project $X_{cen}$ onto $\hat{\mathcal{U}}_m^{ARG}$ using the basis obtained in Step 1.
    \State \textbf{Step 3.} Perform PCA on the projected data to obtain ARG PC variances, directions and scores.

\end{algorithmic}
\end{algorithm}

\end{document}